\documentclass[11pt]{article}

\usepackage{amssymb,amsmath,amsfonts}
\usepackage{graphicx,color,enumitem}
\usepackage{amsthm} 
\usepackage{bm}
\usepackage[round]{natbib}
\usepackage{mathtools}
\usepackage{comment}
\usepackage{geometry}

\usepackage{subcaption}

\usepackage[colorinlistoftodos, textwidth=4cm, shadow]{todonotes}
\RequirePackage[colorlinks,citecolor=blue,urlcolor=blue]{hyperref}

\usepackage{tikz}
\tikzstyle{vertex}=[circle, draw, inner sep=2pt, fill=white]

\newcommand{\E}{{\mathsf E}}
\newcommand{\F}{{\mathbb F}}

\renewcommand{\P}{{\mathsf P}}
\newcommand{\Q}{{\mathsf Q}}

\newcommand{\R}{{\mathbb R}}
\renewcommand{\S}{{\mathbb S}}
\newcommand{\N}{{\mathbb N}}

\newcommand{\Fcal}{{\mathcal F}}
\newcommand{\Gcal}{{\mathcal G}}

\newcommand{\Ocal}{{\mathcal O}}
\newcommand{\Pcal}{{\mathcal P}}

\DeclareMathOperator{\rk}{rank}

\newcommand{\fdot}{{\hspace{1pt}\cdot\hspace{1pt}}}

\DeclareMathOperator{\diam}{diam}
\DeclareMathOperator{\ri}{ri}
\DeclareMathOperator{\rbd}{rbd}
\DeclareMathOperator{\aff}{aff}

\DeclareMathOperator*{\essinf}{ess\,inf}
\newcommand{\essi}[1]{#1\text{-}\essinf}

\DeclareMathOperator{\tr}{tr}

\DeclareMathOperator{\interior}{int}
\DeclareMathOperator{\cl}{cl}

\newtheorem{theorem}{Theorem}

\newtheorem{corollary}[theorem]{Corollary}

\newtheorem{example}[theorem]{Example}
\newtheorem{lemma}[theorem]{Lemma}
\newtheorem{proposition}[theorem]{Proposition}
\newtheorem{remark}[theorem]{Remark}

\numberwithin{equation}{section}
\numberwithin{theorem}{section}

\DeclareMathOperator{\diag}{diag}

\begin{document}

\title{Minimum curvature flow and martingale exit times\footnote{Acknowledgement: Part of this research was completed while we visited I.~Karatzas at Columbia University, whom we thank for his hospitality. J.R. is also grateful 
to FIM at ETH Zurich for hosting.
We would like to thank F.~Da Lio, R.~Kohn, and M.~Shkolnikov for helpful discussions. We are especially grateful to I.~Karatzas and M.~Soner for  pointing out relevant references in the literature, stimulating discussions, and useful suggestions. J.R.~acknowledges financial
support from the EPSRC Research Grant EP/W004070/1.}}
\author{Martin Larsson\thanks{Department of Mathematical Sciences, Carnegie Mellon University, Wean Hall, 5000 Forbes Ave, Pittsburgh, Pennsylvania 15213, USA, \url{martinl@andrew.cmu.edu}.}
\and Johannes Ruf\thanks{Department of Mathematics, London School of Economics and Political Science, Columbia House, London, WC2A 2AE, UK, \url{j.ruf@lse.ac.uk}.}}
\date{January 11, 2022}

\maketitle

\begin{abstract}
We study the following question: What is the largest deterministic amount of time $T_*$ that a suitably normalized martingale $X$ can be kept inside a convex body $K$ in $\R^d$? We show, in a viscosity framework, that $T_*$ equals the time it takes for the relative boundary of $K$ to reach $X(0)$ as it undergoes a geometric flow that we call (positive) minimum curvature flow. This result has close links to the literature on stochastic and game representations of geometric flows. Moreover, the minimum curvature flow can be viewed as an arrival time version of the Ambrosio--Soner codimension-$(d-1)$ mean curvature flow of the $1$-skeleton of $K$. Our results are obtained by a mix of probabilistic and analytic methods.

\bigskip

\textbf{MSC 2020 Classification:} 93E20; 35J60; 49L25
\bigskip

\textbf{Keywords:} Curvature flow, Stochastic control, Viscosity solutions
\end{abstract}

\section{Introduction and main results}\label{S_main}

Let $d\ge2$ and let $K\subset\R^d$ be a convex body, i.e.\ a nonempty compact convex set. If $X=(X_1,\ldots,X_d)$ is a $d$-dimensional continuous martingale that starts inside $K$ and whose quadratic variation satisfies $\tr\langle X\rangle(t)=\langle X_1\rangle(t)+\cdots+\langle X_d\rangle(t) \equiv t$, then $X$ eventually leaves $K$. What is the maximal {\em deterministic} lower bound $T_*$ on the exit time, across all such martingales $X$? The answer is linked to the evolution of the (relative) boundary of $K$ as it undergoes a geometric flow that we refer to as {\em minimum curvature flow:} $T_*$ is equal to the lifetime of this flow. The minimum curvature flow resembles the well-known mean curvature flow, in particular its version in codimension $d-1$ introduced by \cite{amb_son_96}. Our goal is to develop the connection between the exit time problem and the minimum curvature flow in detail.

Our original motivation comes from a long-standing problem in mathematical finance, namely to characterize the worst-case time horizon for so-called {\em relative arbitrage}. In a suitably normalized setup, the answer turns out to be precisely $T_*$, with $K$ being the standard $d$-simplex. We do not discuss this connection further here; instead we provide full details in the companion paper \cite{lar_ruf_20}. Let us however emphasize that this application motivates us to consider convex bodies $K$ with nonsmooth boundary.

To give a precise description of our main results, let $X$ denote the coordinate process on the Polish space $\Omega=C(\R_+,\R^d)$ of all continuous trajectories in $\R^d$ with the locally uniform topology. Thus $X(t,\omega)=\omega(t)$ for all $\omega\in\Omega$ and $t\in\R_+$. Write $\Pcal(\Omega)$ for the set of all probability measures on $\Omega$ with the topology of weak convergence. For each $x\in\R^d$, define
\[
\Pcal_x = \left\{ \P \in \Pcal(\Omega)\colon \text{$X$ is a $\P$-martingale and $\P(X(0)=x )=\P(\tr\langle X\rangle(t)\equiv t)=1$}\right\},
\]
where the martingale property is understood with respect to the (raw) filtration generated by $X$.  
We always take $K\subset\R^d$ to be compact, but not necessarily convex unless explicitly stated. The first exit time from $K$ is
\begin{equation} \label{eq: tauK}
\tau_K = \inf\{t\ge0\colon X(t) \notin K\},
\end{equation}
and we are interested in computing the value function
\begin{equation}\label{eq_vxD}
v(x) = \sup_{\P\in\Pcal_x} \essi{\P} \tau_K.
\end{equation}
This is the largest deterministic almost sure lower bound on the exit time $\tau_K$ across all martingale laws $\P\in\Pcal_x$.

Our first result states that the value function solves a PDE with (degenerate) elliptic nonlinearity
\begin{equation}\label{MCF_FpM}
F(p,M) = \inf\left\{ -\frac12\tr(aM)\colon \text{$a\succeq 0$, $\tr(a)=1$, $ap=0$}\right\},
\end{equation}
where $a$ ranges through all symmetric matrices of appropriate size, and $a\succeq 0$ refers to the positive semidefinite order. The theorem uses the notion of viscosity solution, which is reviewed in Section~\ref{S_viscosity_soln} where also the proof is given.

\begin{theorem}\label{T_usc_vs}
Let $d\ge2$ and suppose $K$ is compact, but not necessarily convex. The value function $v$ is an upper semicontinuous viscosity solution to the nonlinear equation
\begin{equation}\label{PDE_1}
\text{$F(\nabla u,\nabla^2u)=1$}
\end{equation}
in $\interior(K)$ with zero boundary condition (in the viscosity sense).
\end{theorem}

The value function is always an \emph{upper semicontinuous} viscosity solution. As our next result shows, it is actually the unique viscosity solution in this class, provided that $K$ satisfies a certain additional condition. This condition holds for all strictly star-shaped compact sets, in particular for all convex bodies with nonempty interior. Our condition is however more general than that; see Example~\ref{E_T_lambda}. We also show that uniqueness may fail for star-shaped but not strictly star-shaped domains; see Example~\ref{E_non_strict_star}. This answers a question of \citet[Section~1.8]{koh_ser_06}. The proof of the following uniqueness theorem is given in Section~\ref{S_comparison}, and follows from a comparison principle proved there, Theorem~\ref{T_comparison}.

\begin{theorem}\label{T_uniqueness}
Let $d\ge2$ and suppose $K$ is compact. Assume there exist invertible affine maps $T_\lambda$ on $\R^d$, parameterized by $\lambda\in(0,1)$, such that $T_\lambda(K)\subset\interior(K)$ and $\lim_{\lambda\to1}T_\lambda=I$ (the identity). Then the value function $v$ is the unique upper semicontinuous viscosity solution to \eqref{PDE_1} in $\interior(K)$ with zero boundary condition (in the viscosity sense).
\end{theorem}

\begin{remark}
We point out that this uniqueness result is designed to handle the non-smooth convex domains that arise in the financial applications of interest. There are however other natural domains that are not covered by this result, such a various non-convex domains with smooth boundary. Proving comparison theorems (and hence uniqueness results) for such domains is an interesting problem which we do not consider here; see however \citet{son_86_I,son_86_II,bar_rou_sou_99,bar_lio_04}.
\end{remark}

Theorem~\ref{T_uniqueness} characterizes the value function even in cases where it is not continuous. In fact, we will give examples showing that the value function may be discontinuous even when $K$ is a convex body.

Before describing this and related results, we briefly discuss links to the existing literature and the connection to geometric flows.

Our results tie in with a well established literature on stochastic representations of geometric PDEs, initiated by  \cite{BCQ:01} and \cite{son_tou_02,son_tou_02b,son_tou_03}.
In particular, Soner and Touzi introduced the notion of \emph{stochastic target problem} and based their analysis on an associated dynamic programming principle; see also \citet{Bouchard:Vu}.  Part of our analysis can be cast in the language of stochastic target problems, and this connection is described further in Remark~\ref{R_stoch_target}.

The control problem \eqref{eq_vxD} is formulated over an infinite time horizon.   As a result, our PDE is elliptic rather than parabolic, and, as explained next, the solution acquires the interpretation of {\em arrival time} of an evolving surface. 
 This is reminiscent of the two-person deterministic game introduced by \cite{spe_77} and linked to the positive curvature flow by \cite{koh_ser_06}.
In a similar spirit there are also the works of \cite{Peres:2009} on the tug-of-war game and infinity Laplacian, and more recently \cite{Drenska:2020} and \cite{calder2018limit}.

The geometric meaning of \eqref{PDE_1} is most clearly conveyed by reasoning as in Section~1.2 of \citet{koh_ser_06}. This is standard in the literature on geometric flows and paraphrased here for convenience. Let $K$ be strictly convex with smooth boundary $\partial K$. Suppose we are given a family $\{\Gamma_t\colon t\ge0\}$ of smooth convex surfaces with $\Gamma_0=\partial K$, that evolve with normal velocity equal to (half) the smallest principal curvature at each point $x\in\Gamma_t$. It is natural to call this {\em minimum curvature flow}, by analogy with mean curvature flow whose normal velocity is the average curvature.

Let $u$ be the {\em arrival time function:} for each $x\in K$, $u(x)$ is the time it takes the evolving front to reach $x$ (we assume the front passes through each point in $K$ exactly once.) Thus $\Gamma_t=\{x\colon u(x)=t\}$ is a level surface of $u$, and the gradient $\nabla u(x)$ is a normal vector at $x$. If $\nabla u(x)\ne0$, the minimal principal curvature of $\Gamma_t$ at $x$ is the smallest value of
\[
-\frac{y^\top \nabla^2 u(x)y}{|\nabla u(x)|}
\]
as $y$ ranges over all tangent unit vectors: $|y|=1$ and $y^\top \nabla u(x)=0$.\footnote{Indeed, if $\gamma\colon\R\to\Gamma_t$ is a smooth geodesic curve with unit speed such that $\gamma(0)=x$ and $\gamma'(0)=y$, then $\nabla u(x)^\top\gamma''(0)+y^\top\nabla^2u(x)y=0$ and $\gamma''(0)=k\frac{\nabla u(x)}{|\nabla u(x)|}$, where $k$ is the curvature of $\gamma$ at $0$.}  On the other hand, since $u(x)$ is the arrival time, the speed of normal displacement at $x$ is $1/|\nabla u(x)|$. We therefore expect $u$ to satisfy
\begin{equation}\label{eq_infyF}
\inf\left\{ -\frac12y^\top \nabla^2 u(x)y\colon \text{$|y|=1$, $y^\top\nabla u(x)=0$}\right\} = 1,
\end{equation}
at least at points where $\nabla u\ne0$. It is not hard to check that this is precisely \eqref{PDE_1}. In the planar case $d=2$, $\Gamma_t$ has only one principle curvature direction, and \eqref{eq_infyF} reduces to the well-known arrival time PDE for the mean curvature flow,
\[
\frac{1}{|\nabla u|} = - \frac12 {\rm div}\left(\frac{\nabla u}{|\nabla u|}\right).
\]

\begin{remark}
Let us outline how the minimum curvature flow can be constructed rigorously using the \emph{level set method} of \citet{osh_set_88,che_gig_got_91,eva_spr_91} and then linked to \eqref{eq_vxD} and \eqref{PDE_1}. Fix a time horizon $T > \max_{x \in K} v(x)$ and consider the geometric parabolic equation
\[
\partial_t U + F(\nabla U, \nabla^2 U) = 0 \text{ in } (0,T] \times \R^d
\]
with an initial condition $U_0(x)$ that is positive on $\interior(K)$, negative on $K^c$, and constant, say equal to $-1$, outside some large compact set. \citet[Theorems~6.7--6.8]{che_gig_got_91} yields existence and uniqueness of a continuous solution $U(t,x)$ of the initial value problem. One now defines the evolving front of the minimum curvature flow at time $t$ to be the boundary of the superlevel set, $\partial\{x \colon U(t,x) > 0\}$. The time $u(x) = \inf\{t \colon U(t,x) < 0\}$ at which the front passes through $x \in K$ can then, under suitable conditions, be shown to be an upper semicontinuous viscosity solution of the elliptic equation \eqref{PDE_1}. If uniqueness holds for this equation, for instance if Theorem~\ref{T_uniqueness} is applicable, it follows that $u$ actually coincides with the value function $v$ in \eqref{eq_vxD}.
\end{remark}

The link to the control problem \eqref{eq_vxD} can be understood as follows. Proceeding informally, we assume a $C^2$ solution $u$ of \eqref{PDE_1} with $u=0$ on $\partial K$ is given. By It\^o's formula,
\begin{equation}\label{eq_intro_1001}
0 = u(X(\tau_K)) = u(x) + \int_0^{\tau_K} \nabla u(X(t))^\top dX(t) + \frac12 \int_0^{\tau_K} \tr(a(t)\nabla^2 u(X(t)))dt
\end{equation}
under any law $\P\in\Pcal_x$, where $a(t)$ is the derivative of the quadratic variation of $X$ and satisfies $\tr(a(t))\equiv1$. The discussion of minimum curvature flow suggests that optimally, $X$ should fluctuate tangentially to the level surfaces of $u$, that is, $a(t)\nabla u(X(t))\equiv0$. Then, due to the definition \eqref{MCF_FpM} of $F$ and since $u$ solves \eqref{PDE_1},
\begin{equation}\label{eq_intro_11}
\frac12\tr(a(t)\nabla^2 u(X(t))) \le -F(\nabla u(X(t)),\nabla^2u(X(t))) = -1.
\end{equation}
Combining \eqref{eq_intro_1001} and \eqref{eq_intro_11} leads to
\[
0 = u(x) + \frac12 \int_0^{\tau_K} \tr(a(t)\nabla^2 u(X(t)))dt \le u(x) - \tau_K,
\]
showing that $\tau_K\le u(x)$. If $a(t)$ maximizes the left-hand side of \eqref{eq_intro_11}, we have equality and expect that $u$ coincides with the value function. Still heuristically, this happens when $X$ fluctuates only along the minimal principle curvature directions of the level surfaces of $u$. This minimizes the speed at which $X$ moves ``outwards'' toward $\partial K$, and maximizes the amount of time $X$ spends in $K$.

This discussion suggests that optimally, $X$ lies on the evolving front of the time-reversed minimum curvature flow. More precisely, before exiting $K$, one expects that $X$ satisfies $v(X(t))=v(x)-t$ under some optimal law $\P\in\Pcal_x$, $x\in K$. Theorem~\ref{T_smooth_case} below shows that this is true if $K$ is a polytope and $v$ sufficiently regular. It is however false in general, even if $v$ is smooth; see Example~\ref{E_target_superlevel}. 

In the case where $K$ is not convex, we get a somewhat different flow. Similarly to the positive curvature flow of \cite{koh_ser_06}, it is now the positive part of the minimum principal curvature that determines the speed of the flow.

We now return to our main results, and focus on the case where $K$ is a convex body. Theorems~\ref{T_usc_vs} and~\ref{T_uniqueness} yield upper semicontinuity of the value function $v$ and characterize it as a viscosity solution of \eqref{PDE_1} with zero boundary condition (in the viscosity sense). If $K$ has empty interior we simply apply these results in the affine span of $K$. The following result is a combination of Proposition~\ref{P_quasi_concave} and Lemma~\ref{L_Kvzero} in Section~\ref{S_convex}.

\begin{theorem}
Let $d\ge2$ and suppose $K$ is a convex body. Then the value function $v$ is quasi-concave, vanishes on all faces of $K$ of dimension zero and one, and is strictly positive elsewhere in $K$.
\end{theorem}

In particular, if $K$ is strictly convex, then all its boundary faces have dimension zero, and $v$ vanishes everywhere on $\partial K$. Because of upper semicontinuity, this implies that it is continuous at $\partial K$. In fact, Theorem~\ref{T_polytope} below shows that $v$ is continuous everywhere in this case.

However, many convex bodies $K$ have boundary faces of higher dimension. In this case $v$ does \emph{not} vanish everywhere on $\partial K$. This includes the standard $d$-simplex appearing in our motivating financial application. Additionally, and more subtly, there are convex bodies for which the value function is actually {\em discontinuous}. This is because in dimension $d\ge4$, there are convex bodies that admit boundary points $x_n$, all contained in $1$-dimensional boundary faces, whose limit $\bar x=\lim_n x_n$ lies in the relative interior of a $2$-dimensional boundary face; see Example~\ref{E_v_dicontinuous}. For such points, $v(x_n)=0$ but $v(\bar x)>0$, so continuity fails. This is in sharp contrast to the more familiar case of mean curvature flow, where the arrival time function is continuous for any convex initial surface; see \citet[Theorem~7.4]{eva_spr_91} and \citet[Theorem~5.5]{eva_spr_92}.

We prove continuity under the following regularity condition on the geometry of $K$. We require that the {\em $k$-skeletons}, defined by
\begin{equation}\label{F1Fd}
\text{$\Fcal_k = $ union of all faces of $K$ of dimension at most $k$,}
\end{equation}
be closed for $k=1,\ldots,d$ (but not for $k=0$, thus the set of extreme points need not be closed.) This condition is a weakening of a notion from convex geometry called {\em stability}, which is equivalent to all the $k$-skeletons being closed, including the $0$-skeleton; see e.g.\ \cite{pap_77} and \citet[p.~78]{sch_14}. Actually the $d$-, $(d-1)$- and $(d-2)$-skeletons of a convex body are always closed, so this does not have to be assumed separately; see Lemma~\ref{L_skeletons_closed}.

The upshot is the following result, which is applicable in a number of interesting situations. In particular, it covers all convex bodies in $\R^3$, all polytopes in arbitrary dimension, and all convex bodies whose boundary faces all have dimension zero or one. It is a rewording of Theorem~\ref{T_190603} in Section~\ref{S_convex}, and is proved using probabilistic arguments based on the control formulation \eqref{eq_vxD}.

\begin{theorem}\label{T_polytope}
Let $d\ge2$ and suppose $K$ is a convex body with $\Fcal_k$ closed for $1\le k\le d-3$. Then the value function $v$ is continuous on $K$.
\end{theorem}

The fact that $v$ vanishes only at the $1$-skeleton $\Fcal_1$ (the extreme points and lines), but not elsewhere in $K$, suggests that \eqref{PDE_1} describes a geometric flow also of $\Fcal_1$, not only of $\partial K$. This flow of $\Fcal_1$ is the {\em codimension-$(d-1)$ mean curvature flow} of \cite{amb_son_96}, although here the initial set $\Fcal_1$ need not be a one-dimensional curve.

To spell this out, for any symmetric matrix $A$ and eigenvector $p$ of $A$, let $\lambda_{\rm min}(A,p)$ denote the smallest eigenvalue of $A$ corresponding to an eigenvector orthogonal to $p$. Then \eqref{eq_infyF} states that
\begin{equation}\label{eq_PDE_lambda_intro}
\lambda_{\rm min}\left(-\frac12 P_{\nabla u(x)} \nabla^2u(x) P_{\nabla u(x)}, \nabla u(x)\right) = 1,
\end{equation}
where
\[
P_{\nabla u} = I - \frac{\nabla u\nabla u^\top}{|\nabla u|^2}.
\]
Modulo sign conventions and the factor $1/2$, the left-hand side of \eqref{eq_PDE_lambda_intro} is precisely the operator used by \cite{amb_son_96}. In fact, the function $V(t,x)=t-v(x)$, where $v$ is the value function in \eqref{eq_vxD}, solves their parabolic equation on $K$ with initial condition $V(0,x)=-v(x)$, whose zero set (in $K$) is the $1$-skeleton $\Fcal_1$. This suggests interpreting the minimum curvature flow of $\partial K$ as a codimension-$(d-1)$ mean curvature flow of $\Fcal_1$.

This perspective is particularly compelling when $K$ is a polytope: $\Fcal_1$ is then a finite union of closed line segments and thus one-dimensional, albeit with ``branching''. In this case, the one-dimensional initial contour instantly develops higher-dimensional features as it evolves under the flow, and eventually becomes a closed hypersurface. This is illustrated schematically in Figure~\ref{F_fattening}, where $K$ is the standard $3$-simplex.

\begin{figure}[h]
\centering
\begin{subfigure}[t]{.27\textwidth}
\centering
\includegraphics[width=0.95\linewidth]{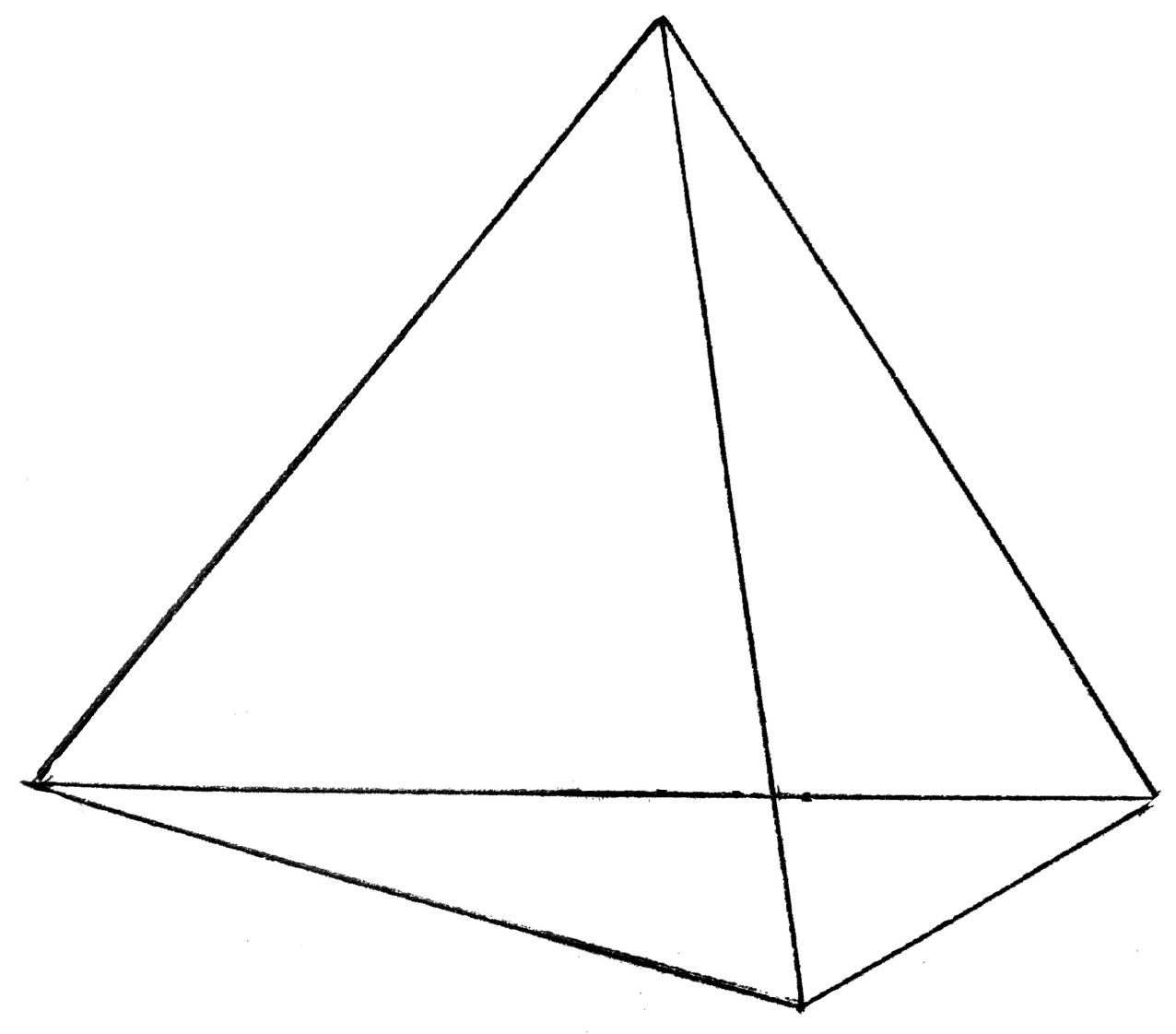}
\end{subfigure}
\hspace{0.03\textwidth}
\begin{subfigure}[t]{.27\textwidth}
\centering
\includegraphics[width=\linewidth]{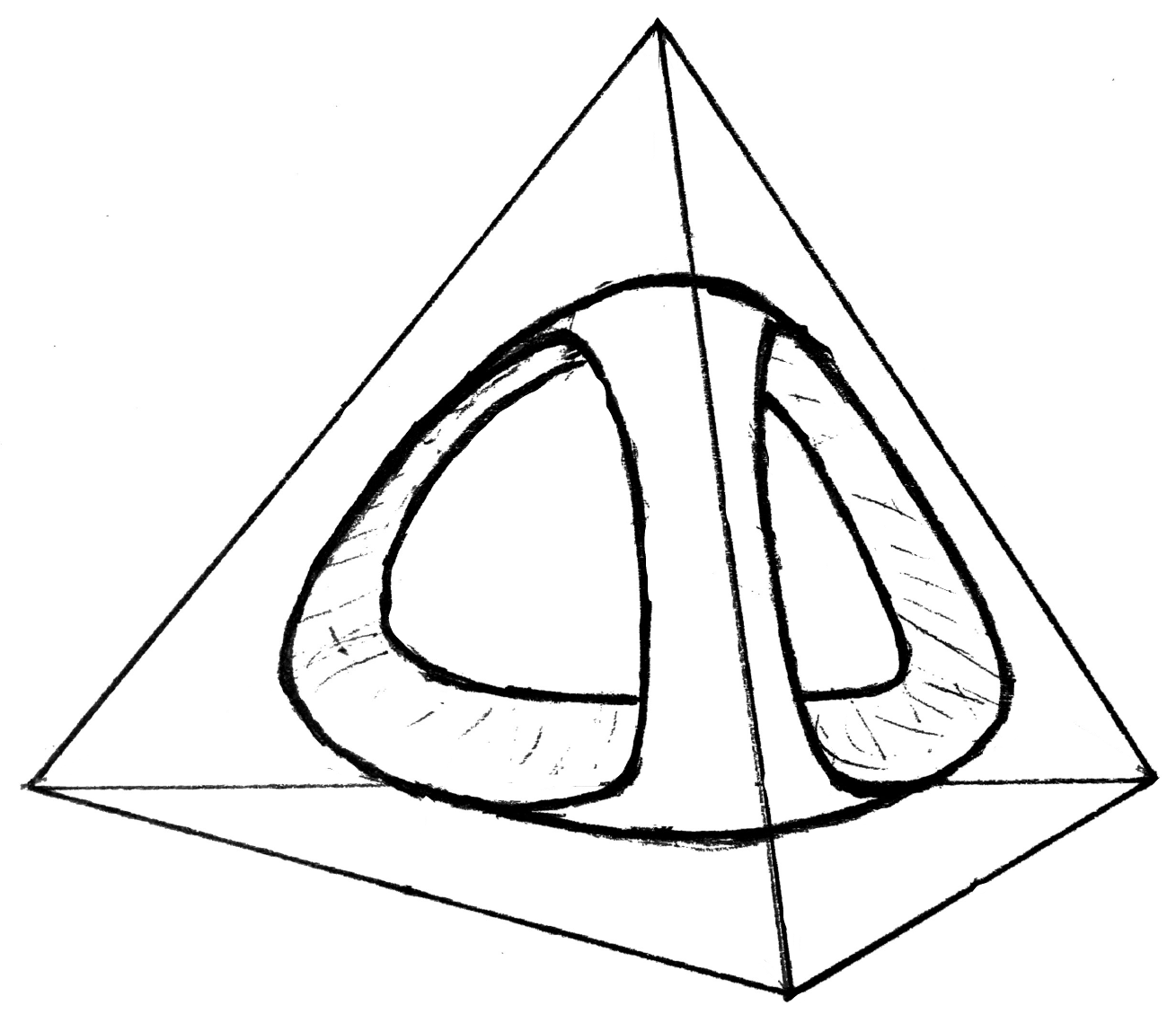}
\end{subfigure}
\hspace{0.03\textwidth}
\begin{subfigure}[t]{.27\textwidth}
\centering
\includegraphics[width=\linewidth]{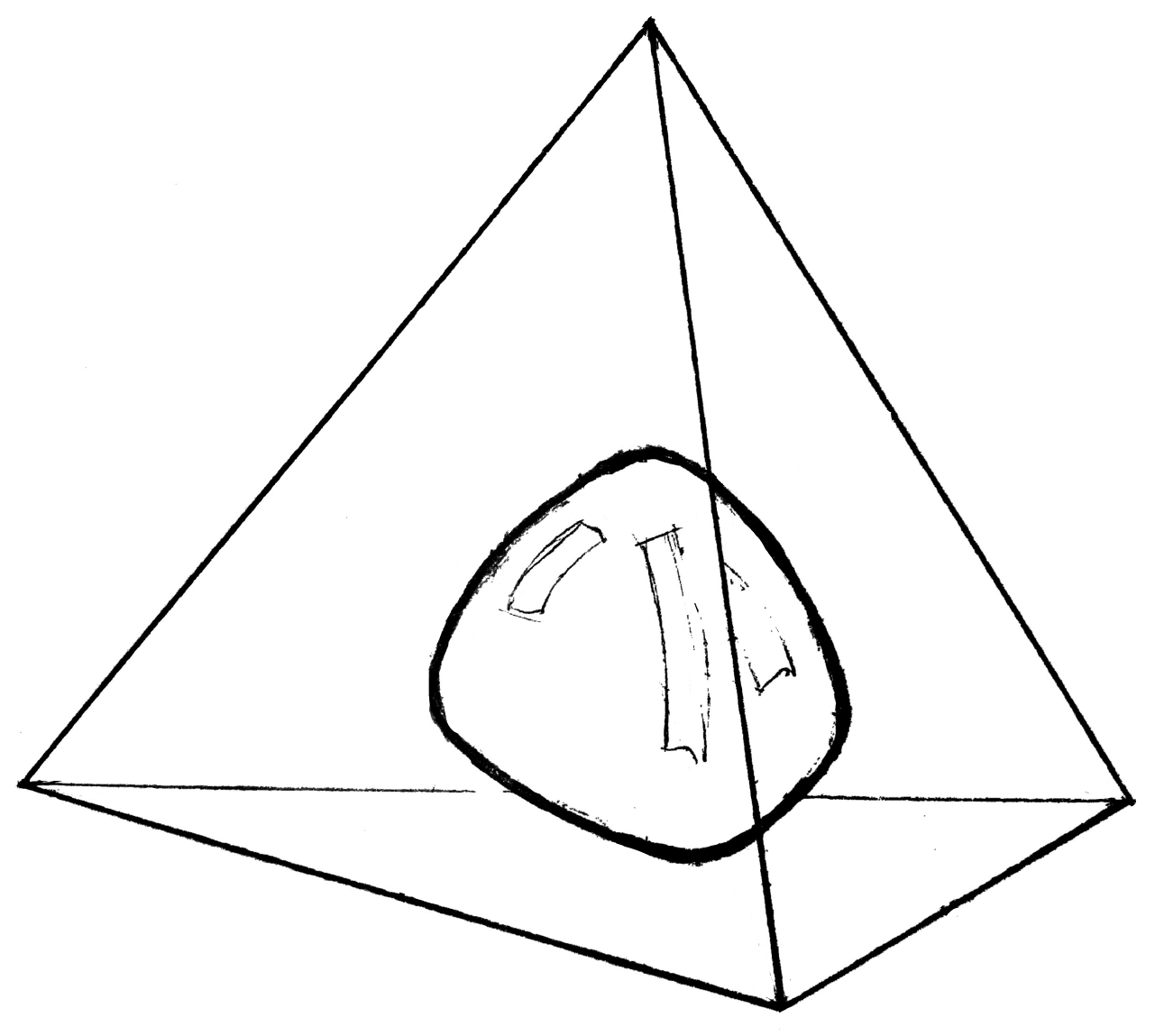}
\end{subfigure}
\caption{Schematic illustration of the minimum curvature flow of the $3$-simplex, regarded as codimension-2 mean curvature flow of its 1-skeleton as initial contour. In the second and third panel, the 1-skeleton is still shown for reference.}
\label{F_fattening}
\end{figure}

Returning to the minimum curvature flow as a flow of surfaces starting from $\partial K$, we see that points inside two- and higher dimensional faces remain stationary for some period of time. This behavior is analogous to the behavior of mean curvature flow of {\em non-convex} contours; see \citet[Figure~4]{koh_ser_06} for an illustration. We thank R.~Kohn for pointing this out to us. A similar phenomenon occurs for the Gauss curvature flow; see \citet{ham_94,cho_eva_ish_99,das_lee_04}.

We do not have much information about the regularity of the value function $v$ in general, beyond the continuity assertion in Theorem~\ref{T_polytope} and the counterexample in Example~\ref{E_v_dicontinuous}. An exception is the planar case $d=2$, where we recover the standard mean curvature flow. In this case, for $K$ strongly convex with smooth boundary, \cite{koh_ser_06} proved that $v$ is $C^3$ (see also \cite{hui_93} for an earlier proof that $v$ is $C^2$). In general, let us {\em assume} that $v$ is $C^2$ inside each face of $K$, with just one critical point. If in addition $K$ has at most countably many faces, it is then possible to construct optimal solutions of \eqref{eq_vxD} where the intuitive notion that $X$ should fluctuate tangentially to, and remain on, the level surfaces of $v$ becomes rigorous.

\begin{theorem}\label{T_smooth_case}
Let $d\ge2$ and let $K$ be a convex body with at most countably many faces. Assume the value function $v$ lies in $C^2(K)$. Assume also that in each face $F$ of dimension at least two, either $v$ has no critical point, or $v$ has one single critical point which additionally is a maximum. Then for every $\bar x\in K$ there is an optimal solution $\P\in\Pcal_{\bar x}$ under which $v(X(t))=v(\bar x)-t$ for all $t<\tau_K$. In particular,
\[
\int_0^t \nabla v(X(s))^\top dX(s) = 0, \quad t<\tau_K,
\]
and $X$ lies on the evolving front of the time-reversed minimum curvature flow in the sense that $X(t)\in\Gamma_{v(\bar x)-t}$, where $\Gamma_t=\{x\colon v(x)=t\}$, until it leaves $K$.
\end{theorem}

The meaning of $C^2(K)$ and the notion of a critical point is explained in Section~\ref{S_v_smooth}, where also the proof is given. The basic idea is to observe that $v$ satisfies \eqref{eq_PDE_lambda_intro} classically at non-critical points. In particular, by definition of eigenvalue, the matrix
\[
H(x) = \frac12 P_{\nabla v(x)} \nabla^2v(x) P_{\nabla v(x)} + I
\]
is singular at all such points, so is of rank at most $d-1$. This can be used to construct a martingale law $\P\in\Pcal_{\bar x}$ under which $H(X(t))d\langle X\rangle(t)=0$. This turns out to imply $\nabla v(X(t))^\top dX(t)=0$ and then $dv(X(t))=-dt$. This is essentially the desired conclusion. Some effort is needed to construct $\P$, basically because the Moore--Penrose inverse $H(x)^+$ of $H(x)$ is no longer continuous in $x$. Moreover, $X$ is obtained by constructing martingales on each face of $K$ separately and then ``gluing'' these martingales together. This introduces some technical hurdles, and explains why the proof is somewhat lengthy.

As an illustration, and for later use, we give a simple example where the value function $v$ is known explicitly and happens to be smooth on $K$; see also \cite{str_71} and \cite{FKR:16}.

\begin{example}  \label{Ex:190511}
Let $d\ge2$ and let $K=\{x\in\R^d\colon |x|\le r\}$ be the centered closed ball of radius $r>0$. In this case, $v(x)=r^2-|x|^2$ for all $x\in K$. To see this, choose any $x \in K$ and $\P \in \Pcal_x$. We have
\[
|X(t)|^2 = |x|^2 + 2\int_0^t X(s)^\top dX(s) + t, \quad t\ge0.
\]
Evaluating at $t=\tau_K\wedge n$, taking expectations, and letting $n\to\infty$, one obtains $\E[\tau_K]=r^2-|x|^2$. In particular, this shows that $X$ escapes from any bounded set in finite time, $\P$-a.s. Moreover, since of course $\essi{\P} \tau_K\le\E[\tau_K]$, we get $v(x) \le |x|^2-r^2$. In fact, we have equality. Indeed, let $\P\in\Pcal_x$ be the law under which $X_3,\ldots,X_d$ are constant and $(X_1,X_2)^\top$ satisfies
\[
d\begin{pmatrix}X_1(t)\\X_2(t)\end{pmatrix} = \frac{1}{\sqrt{X_1(t)^2 + X_2(t)^2}}  \begin{pmatrix}X_2(t)\\-X_1(t)\end{pmatrix} dW(t),
\]
where $W$ denotes a one-dimensional Brownian motion. Such a probability measure $\P$ always exists, even if $x = 0$; see Lemma~\ref{L_weaksolution}. An application of It\^o's formula now yields $\tau_K = r^2 - |x|^2$, $\P_x$-a.s. We deduce that $v(x)=r^2-|x|^2$ for all $x\in K$. Furthermore, it is straightforward to verify that $v$ satisfies \eqref{PDE_1} with boundary condition $v = 0$ on $\partial K$.
\end{example}

The reasoning in Example~\ref{Ex:190511} directly yields the following upper bound on $v$.

\begin{lemma}\label{L_E[theta]}
If $K$ is compact, $x\in K$, and $\P\in\Pcal_x$, then $\E[\tau_K]\le r^2$, where $r$ is the radius of the smallest ball containing $K$. In particular, $\tau_K<\infty$, $\P$-a.s., and the value function defined in \eqref{eq_vxD} satisfies $v(x)\le r^2$ for all $x$.
\end{lemma}

The rest of the paper is organized as follows. Section~\ref{S_v_DPP} develops a number of general properties of the value function, as well as illustrative examples. In particular, a dynamic programming principle is proved. In Section~\ref{S_viscosity_soln} we prove Theorem~\ref{T_usc_vs} that the value function is a viscosity solution. In Section~\ref{S_comparison} we prove Theorem~\ref{T_comparison}, a comparison principle for viscosity solutions of \eqref{PDE_1}, and use it to deduce Theorem~\ref{T_uniqueness}. In Section~\ref{S_convex} we focus on the case where $K$ is a convex body, and establish in particular Theorem~\ref{T_190603} on continuity of the value function. In Section~\ref{S_v_smooth} we prove Theorem~\ref{T_smooth_case}.

We end with a technical remark regarding filtrations and stopping times. Whenever $X$ is said to be a martingale, this is understood with respect to its own filtration $\F^X=(\Fcal^X_t)_{t\ge0}$ where $\Fcal^X_t=\sigma(X_s, s\le t)$. In this case, $X$ is also a martingale for the right-continuous filtration $\F^X_+$ consisting of the $\sigma$-algebras $\bigcap_{u>t}\Fcal^X_u$, and similarly for the filtrations obtained by augmenting $\F^X$ and $\F^X_+$ with nullsets. In particular, results such as the stopping theorem are applicable with $\tau_K$ in \eqref{eq: tauK}, which is an $\F^X_+$-stopping time but not an $\F^X$-stopping time.

\section{The value function and dynamic programming}\label{S_v_DPP}

The purpose of this section is to establish a number of properties of the value function, in particular a dynamic programming principle. Throughout this section, $K$ is compact but not necessarily convex.

\begin{lemma}\label{L_usc_tauK}
The maps $\omega \mapsto \tau_K(\omega)$ from $\Omega$ to $[0,\infty]$ and $\P\mapsto\essi{\P}  \tau_K$ from $\Pcal(\Omega)$ to $[0,\infty]$ are upper semicontinuous, where $\tau_K$ is the first exit time of $K$, given in \eqref{eq: tauK}.
\end{lemma}

\begin{proof}
We claim that $\omega\mapsto\tau_K(\omega)$ is upper semicontinuous on $\Omega$. To see this, let $\omega_n,\omega$ satisfy $\tau_K(\omega)<\infty$ and $\omega_n\to\omega$ locally uniformly. Consider $\varepsilon>0$ such that $\omega(\tau_K(\omega)+\varepsilon)\notin K$. Then for all large $n$, we have $\omega_n(\tau_K(\omega)+\varepsilon)\notin K$, and hence $\tau_K(\omega_n)\le\tau_K(\omega)+\varepsilon$. Thus $\limsup_n\tau_K(\omega_n)\le\tau_K(\omega)+\varepsilon$. This proves upper semicontinuity of $\tau_K$ since $\varepsilon>0$ can be chosen arbitrarily small.

Next, for every $\lambda>0$ the Portmanteau theorem yields that the map
\[
\P\mapsto f_\lambda(\P)=-\frac1\lambda\log\E_\P[e^{-\lambda\tau_K}]
\]
from $\Pcal(\Omega)$ to $[0,\infty]$ is upper semicontinuous. Then so is $\P \mapsto \inf_{\lambda>0}f_\lambda(\P)=\essi{\P} \tau_K$, as required. 
 \end{proof}

\begin{proposition}\label{P_v_proporties}
\begin{enumerate}
\item\label{P_v_proporties_1} $\Pcal_x$ is weakly compact for every $x\in\R^d$;
\item\label{P_v_proporties_2} $v$, given in \eqref{eq_vxD}, is upper semicontinuous and there is a measurable map $x\mapsto\P_x$ from $\R^d$ into $\Pcal(\Omega)$ such that $\P_x$ lies in $\Pcal_x$ and is optimal for all $x\in\R^d$;
\item\label{P_v_proporties_3} $v$ satisfies the following dynamic programming principle: for every $x\in\R^d$ and every $\F^X$-stopping time $\theta$,
\[
v(x) = \sup_{\P\in\Pcal_x} \essi{\P} \{ \theta \wedge\tau_K + v(X(\theta))\bm1_{\theta\le\tau_K}\}.\footnote{Note that $\tau_K<\infty$,  $\P$-a.s.\ for every $\P\in\Pcal_x$ due to Lemma~\ref{L_E[theta]}.}
\]
Moreover, the supremum is attained by any optimal $\P\in\Pcal_x$.
\end{enumerate}
\end{proposition}

\begin{proof}
\ref{P_v_proporties_1}: Consider any $\P\in\Pcal_x$. Fix $s\ge0$ and define $M(t)=|X(t)-X(s)|^2-t+s$ for $t\ge s$. Then $M$ is a $\P$-martingale on $[s,\infty)$ with $\langle M\rangle(t)\le 4\int_s^t |X(u)-X(s)|^2 du$. Thus
\[
\E_\P[\langle M\rangle(t)]\le 4\int_s^t \E_\P[|X(u)-X(s)|^2]du=4\int_s^t (u-s)du=2(t-s)^2,
\]
so that
\[
\E_\P[|X(t)-X(s)|^4] = \E_\P[(M(t)+t-s)^2] \le 2 \E_\P[\langle M\rangle(t)] + 2(t-s)^2 \le 6(t-s)^2.
\]
Kolmogorov's continuity criterion (see \citet{RY:99}, Theorem~I.2.1 and its proof) then gives, for any fixed $T>0$ and $\alpha\in(0,\frac14)$,
\[
\E_\P\left[ \left( \sup_{0\le s<t\le T} \frac{|X(t)-X(s)|}{|t-s|^\alpha} \right)^4 \right] \le c
\]
for some constant $c=c(T,\alpha)$ that does not depend on $\P\in\Pcal_x$. Since H\"older balls are relatively compact in $C([0,T],\R^d)$ by the Arzel\`a--Ascoli theorem, it follows that $\Pcal_x$ is tight and hence relatively compact by Prokhorov's theorem. To see that $\Pcal_x$ is closed, note that the martingale property of both $X$ and $|X|^2-t$ (and hence the property $\tr\langle X\rangle(t)\equiv t$) carries over to weak limits of sequences in $\Pcal_x$. 

\ref{P_v_proporties_2}: First observe that $\Pcal_x$ consists of the pushforwards $(x+\fdot)_*\P$ with $\P\in\Pcal_0$. Thus $v(x)=\sup_{\P\in\Pcal_0}f(x,\P)$, where $f(x,\P)=g((x+\fdot)_*\P)$ and $g(\P)=\essi{\P}  \tau_K$. By Lemma~\ref{L_usc_tauK}, the function $g$ is upper semicontinuous. Since $f$ is the composition of $g$ with the continuous function $(x,\P)\mapsto(x+\fdot)_*\P$ from $\R^d \times \Pcal(\Omega)$ to $\Pcal(\Omega)$, it is also upper semicontinuous.
 Moreover, $\Pcal_0$ is compact by \ref{P_v_proporties_1}. A suitable selection theorem, see e.g.\ \citet[Proposition~7.33]{ber_shr_78}, yields upper semicontinuity of $v$ as well as a measurable map $x\mapsto \Q_x$ from $\R^d$ into $\Pcal_0$ such that $v(x)=f(x,\Q_x)$ for all $x\in\R^d$. Setting $\P_x=(x+\fdot)_*\Q_x$ gives the required map.

\ref{P_v_proporties_3}:
Fix $x\in\R^d$ and an $\F^X$-stopping time $\theta$. We first first fix $\P \in \mathcal P_x$ and prove that
\begin{equation}\label{P_v_proporties_3_eq0}
v(x) \ge \essi{\P} \{\theta\wedge\tau_K + v(X(\theta))\bm1_{\theta\le\tau_K}\}.
\end{equation}
To this end, consider the extended space $\Omega\times\Omega$ with coordinate process $(X,Y)(t,\omega,\tilde\omega)=(\omega(t),\tilde\omega(t))$ and define a law $\P'$ on $(\Omega\times\Omega, \Fcal \otimes \Fcal)$ by $\P'(d\omega,d\tilde\omega)=\P_{X(\theta(\omega),\omega)}(d\tilde\omega)\P(d\omega)$, where
we use the measurable map $\R^d \ni y\mapsto\P_y \in \Pcal$ from \ref{P_v_proporties_2}. 
We now consider the process $X'(t) = X(t)\bm1_{t\le\theta} + Y(t-\theta)\bm1_{t>\theta}$ and let $\Q$ denote the law of $X'$.    Define next $\theta'(\omega,\tilde\omega)=\theta(X'(\omega,\tilde\omega))$; thus $\theta'$ depends on the trajectory of $X'$ like $\theta$ depends on the trajectory of $X$. Since $\theta$ is an $\F^X$-stopping time, and since $X'(t)$ and $X(t)$ coincide for all $t\le\theta$, it follows by Galmarino's test that $\theta'(\omega,\tilde\omega)=\theta(\omega)$ for all $(\omega,\tilde\omega)$; see \citet[Lemma~1.3.3]{SV_multi}. 
Consequently, for all bounded measurable maps $F,G\colon\Omega\to\R$, we have
\begin{align*}
\E_\Q[F(X(\fdot\wedge\theta))G(X(\theta+\fdot))]
&= \E_{\P'}[F(X'(\fdot\wedge\theta'))G(X'(\theta'+\fdot))] \\
&= \E_{\P'}[F(X(\fdot\wedge\theta))G(Y)] \\
&= \E_{\P}[F(X(\fdot\wedge\theta))\E_{\P_{X(\theta)}}[G(X)]].
\end{align*}
Thanks to the definition of $\Q$ we have
\begin{equation}\label{P_v_proporties_3_eq1a}
	\Q\in\Pcal_x,\qquad \Q|_{\Fcal^X_\theta}=\P|_{\Fcal^X_\theta}.
\end{equation}
Furthermore, with the notation $\tau_K(X(\theta+\fdot))=\inf\{t\ge0\colon X(\theta+t)\notin K\}$, one derives the identity
\begin{equation}\label{P_v_proporties_3_eq10113}
\tau_K = \theta\wedge\tau_K + \tau_K(X(\theta+\fdot))\bm1_{\theta\le\tau_K}.
\end{equation}
Finally, the $\Q$-conditional distribution of $X(\theta + \cdot)$ given $\Fcal^X_\theta$ equals the $\P_{X(\theta)}$-distribution of $X(\cdot)$. Since also $\P_y$ is optimal for every $y$, we get
\begin{equation}\label{P_v_proporties_3_eq1b}
\tau_K(X(\theta+\fdot)) \ge v(X(\theta)),\quad \text{$\Q$-a.s.}
\end{equation}
Combining the definition of $v(x)$, \eqref{P_v_proporties_3_eq1a}, \eqref{P_v_proporties_3_eq10113}, and \eqref{P_v_proporties_3_eq1b}, we get
\begin{align*}
v(x) \ge \essi{\Q}  \tau_K &= \essi{\Q}  \{\theta\wedge\tau_K + \tau_K(X(\theta+\fdot))
\bm 1_{ \theta\le\tau_K} \}\\
&\ge \essi{\Q}  \{\theta\wedge\tau_K + v(X(\theta))\bm1_{\theta\le\tau_K}\} \\
&= \essi{\P}  \{\theta\wedge\tau_K + v(X(\theta))\bm1_{\theta\le\tau_K}\}.
\end{align*}
In the last step we used that $\theta\wedge\tau_K$ and $\bm1_{\theta\le\tau_K}$ are $\Fcal^X_\theta$-measurable (even though $\tau_K$ is only an $\F^X_+$-stopping time) and hence have the same law under $\P$ as under $\Q$ due to \eqref{P_v_proporties_3_eq1a}. This proves \eqref{P_v_proporties_3_eq0}.

It remains to prove that
\begin{equation}\label{P_v_proporties_3_eq4}
v(x) = \essi{\P} \{\theta\wedge\tau_K+v(X(\theta))\bm1_{\theta\le\tau_K}\}
\end{equation}
for any optimal $\P\in\Pcal_x$. The proof uses the notion of conditional essential infimum. For a random variable $Y$ and a sub-$\sigma$-algebra $\Gcal\subset\Fcal$, the conditional essential infimum of $Y$ given $\Gcal$ is defined as the largest $\Gcal$-measurable random variable $\P$-a.s.\ dominated by $Y$, denoted by $\essi{\P} \{Y\mid\Gcal\}$. Moreover, if $\{F_\omega\}_{\omega\in\Omega}$ is a regular conditional distribution of $Y$ given $\Gcal$, we have $\essi{\P} \{Y\mid\Gcal\}(\omega)=\essinf F_\omega$ for $\P$-a.e.\ $\omega$, where we set $\essinf F_\omega=\sup\{c \in \R\colon F_\omega([c,\infty))=1\}$. For further details, see \cite{bar_car_jen_03,lar_18_condinf}. 

Now, fix any optimal $\P\in\Pcal_x$. Then, using \eqref{P_v_proporties_3_eq10113}, we get
\begin{equation}\label{P_v_proporties_3_eq6}
v(x) \le \tau_K = \theta\wedge\tau_K + \tau_K(X(\theta+\fdot))\bm1_{\theta\le\tau_K},\quad \text{$\P$-a.s.}
\end{equation}
Next, let $\{Q_\omega\}_{\omega\in\Omega}$ be a regular conditional distribution of $X(\theta+\fdot)$ given $\Fcal^X_\theta$; see \citet[Theorem~1.3.4]{SV_multi}. In particular, $\{F_\omega\}_{\omega\in\Omega}$ with $F_\omega=Q_\omega(\tau_K\in\fdot)$ is then a regular conditional distribution of $\tau_K(X(\theta+\fdot))$ given $\Fcal^X_\theta$. Take now the $\Fcal^X_\theta$-conditional essential infimum in \eqref{P_v_proporties_3_eq6}. Since $\theta\wedge\tau_K$ and $\bm1_{\theta\le\tau_K}$ are $\Fcal^X_\theta$-measurable we get
\begin{align*}
v(x) &\le \theta\wedge\tau_K + \bm1_{\theta\le\tau_K} \essi{\P} \{\tau_K(X(\theta+\fdot))\mid\Fcal^X_\theta\} \\
&= \theta\wedge\tau_K + \bm1_{\theta\le\tau_K}\essinf F_\omega \\
&= \theta\wedge\tau_K + \bm1_{\theta\le\tau_K}\essi{\Q_\omega}\tau_K, \quad \text{$\P$-a.s.}
\end{align*}
One readily verifies that $\Q_\omega\in\Pcal_{X(\theta,\omega)}$ for $\P$-a.e.\ $\omega$. Hence $\essi{\Q_\omega}\tau_K\le v(X(\theta,\omega))$ for $\P$-a.e.\ $\omega$, and we deduce that $v(x)\le\theta\wedge\tau_K+v(X(\theta))\bm1_{\theta\le\tau_K}$, $\P$-a.s. This yields \eqref{P_v_proporties_3_eq4}, and completes the proof of the proposition.
\end{proof}

It is not true in general that, under an optimal law, $X(t)$ is located on the $t$-level surface of the value function, even if the value function is smooth. The following example illustrates this.

\begin{example}\label{E_target_superlevel}
Let $K\subset\R^3$ be the union of the line segment $L=(-1,1)\times\{(0,0)\}$ and the shifted unit discs $(1,0,0)+D$ and $(-1,0,0)+D$ with $D=\{(0,y,z)\colon y^2+z^2\le 1\}$. Thanks to Example~\ref{Ex:190511}, at points $(\pm1,y,z)$ in the shifted discs, the value function is $v(\pm1,y,z)=1-y^2-z^2$. At points $\bar x=(x,0,0)\in L$, the value function is $v(\bar x)=1$. Indeed, $X$ evolves as a Brownian motion along $L$ until it hits $(\pm1,0,0)$. This happens arbitrarily quickly, and at either point the value function is $1$. Thus everywhere in $K$, $v(x,y,z)=1-y^2-z^2$. We see that for $\bar x\in L$, under any optimal $\P\in\Pcal_{\bar x}$ one has $v(X(t))>v(\bar x)-t$ for all $t>0$. Note that in this example, $v$ is very smooth: on $K$ it coincides with a polynomial.
\end{example}

The following result can be viewed as an assertion about {\em propagation of continuity}: if the value function is continuous on a certain set, then it is also continuous on a larger set. Upper semicontinuity, which holds in general due to Proposition~\ref{P_v_proporties}\ref{P_v_proporties_2}, plays an important role. A refined version of this result is crucial in Section~\ref{S_convex}, where $K$ will be a convex body.

\begin{proposition}\label{P:190523}
Let $K$ be compact, and assume $v|_{\partial K}$ is continuous. Then $v|_K$ is continuous.
\end{proposition}

\begin{proof}
Since $v$ is upper semicontinuous by Proposition~\ref{P_v_proporties}\ref{P_v_proporties_2}, since $\partial K$ is compact, and since $v|_{\partial K}$ is continuous by assumption, Lemma~\ref{L_unif_usc} below gives a modulus $\omega$ such that
\begin{equation}\label{eq_03062019_15}
\text{$v(x) \le v(y) + \omega(|x-y|)$ for all $x\in\R^d$ and $y\in\partial K$.}
\end{equation}
Fix $\bar x,\bar y\in K$ and an optimal law $\P\in\Pcal_{\bar x}$. Define the process $Y=X-\bar x+\bar y$ and the $\F^X$-stopping time $\theta=\inf\{t\ge0\colon Y(t)\notin \interior(K)\}$. Note that $\P(\theta<\infty)=1$ by Example~\ref{Ex:190511}. Since $Y(\theta)\in\partial K$, we have from \eqref{eq_03062019_15} that
\[
v(X(\theta)) \le v(Y(\theta)) + \omega(|\bar x - \bar y|), \quad \text{$\P$-a.s.}
\]
We now combine this with two applications of the dynamic programming principle of Proposition~\ref{P_v_proporties}\ref{P_v_proporties_3}.  We get
\begin{align*}
v(\bar x) &= \essi{\P}  \{\theta\wedge\tau_K + v(X(\theta))\bm1_{\theta\le\tau_K}\} \\
&\le \essi{\P}  \{\theta + v(Y(\theta))\} + \omega( |\bar x-\bar y|) \\
&\le v(\bar y) + \omega(|\bar x-\bar y|).
\end{align*}
In the last inequality, the application of the dynamic programming principle uses that the law of $Y$ lies in $\Pcal_{\bar y}$, that $\F^Y=\F^X$, and that $\theta\le\inf\{t\ge0\colon Y(t)\notin K\}$, $\P$-a.s. Since $\bar x,\bar y\in K$ were arbitrary, we deduce that $v|_K$ is uniformly continuous with modulus $\omega$.
\end{proof}

The following lemma is elementary, but crucial for our results on propagation of continuity. This is what allows us to exploit the fact that the value function is always upper semicontinuous.

\begin{lemma}\label{L_unif_usc}
Let $C\subset\R^d$ be a compact set, and let $f\colon\R^d\to\R$ be a function that is upper semicontinuous at every point in $C$. If the restriction $f|_C$ is continuous, then there exists a modulus $\omega$ such that
\[
\text{$f(x) \le f(y) + \omega(|x-y|)$ for all $x\in\R^d$ and $y\in C$.}
\]
\end{lemma}

\begin{proof}
It suffices to pick any $\varepsilon>0$ and exhibit $\delta>0$ such that $f(x)\le f(y)+\varepsilon$ holds whenever $x\in\R^d$, $y\in C$, and $|x-y|<\delta$. Since $f|_C$ is continuous and $f$ is upper semicontinuous at $C$, for every $y\in C$ there exists $\delta_{y}>0$ such that $|f(y)-f(y')|<\varepsilon/2$ and $f(x)< f(y)+\varepsilon/2$ whenever $y'\in C$, $|y-y'|<\delta_{y}$, $x\in\R^d$, $|y-x|<\delta_{y}$. The balls $B(y,\delta_{y}/2)$, $y\in C$, cover $C$. By compactness, there is a finite subcover $B(y_i,r_i)$, $i=1,\ldots,n$, where $r_i=\delta_{y_i}/2$. Define $\delta=\min\{r_1,\ldots,r_n\}$. Suppose $x\in\R^d$, $y\in C$, and $|x-y|<\delta$. Then $y\in B(y_i,r_i)$ for some $i \in \{1, \ldots, n\}$, and hence $|x-y_i|\le|x-y|+|y-y_i|<2r_i\le\delta_{y_i}$ and $|y-y_i|<\delta_{y_i}$. Therefore
\[
f(x)< f(y_i)+\frac{\varepsilon}{2} \le f(y)+|f(y_i)-f(y)|+\frac{\varepsilon}{2}< f(y)+\varepsilon,
\]
as required.
\end{proof}

As mentioned in Section~\ref{S_main}, some of the analysis in this paper can be cast in the language of stochastic target problems. We end this section with a remark detailing this connection. Since this is not used in the analysis to come, we do not give proofs.

\begin{remark} \label{R_stoch_target}
For any $t\in[0,\infty)$, the \emph{target reachability set} when the target is $K$ and the controlled state dynamics is described by $\Pcal_x$ is defined by
\[
V(t) = \{x\in\R^d\colon \text{$\exists \P\in\Pcal_x$ such that $X(t) \in K$ a.s.}\}.
\]
This is a ``time-to-maturity'' version, in a weak formulation, of the definition in \citet[Section~2.4]{son_tou_02}. 
Clearly $V(0)=K$, and one can show that $V(t)=\emptyset$ for all $t>\text{diam}(K)^2/4$.
One expects the following representation of the value function $v$ in \eqref{eq_vxD} in terms of the target reachability set:
\[
v(x) = \sup\{t\ge0\colon x \in V(t)\}, \quad x\in K.
\]
This equality can be shown to hold if $K$ is convex, but there are non-convex examples where it fails. In such cases, one can work with the \emph{obstacle version} of the stochastic target problem, where the reachability set is defined by
\[
W(t) = \{x\in\R^d\colon \text{$\exists \P\in\Pcal_x$ such that $X_s \in K$ for all $s\in[0,t]$ a.s.}\}.
\]
This problem is discussed briefly in Section~7 of \cite{son_tou_02} and further in \cite{bou_vu_10} (where the terminology ``obstacle version'' is introduced). It is straightforward to show that
\[
v(x) = \sup\{T\ge0\colon x\in W(T)\}, \quad x \in K,
\]
regardless of the geometry of $K$. A suitable weak-formulation version of the geometric dynamic programming principle in Theorem~7.1 of \cite{son_tou_02} or Theorem~2.1 of \cite{bou_vu_10} could then be used to derive characterizations of $W(t)$, and hence $v(x)$, in terms of PDEs.
\end{remark}

\section{The value function is a viscosity solution}\label{S_viscosity_soln}

In this section we prove Theorem~\ref{T_usc_vs}, the viscosity solution property, assuming that $d \geq 2$ and that $K$ is compact but not necessarily convex. (We already know from Proposition~\ref{P_v_proporties}\ref{P_v_proporties_2} that $v$ is upper semicontinuous.) A  bounded function $u\colon K \to\R$ is called a \emph{viscosity subsolution} of $F(\nabla u,\nabla^2u)=1$ in $\interior(K)$ if
\[
\left.
\begin{minipage}[c][3em]{.35\textwidth}\center
$(\bar x,\varphi)\in \interior(K) \times C^2(\R^d)$ and \\[1ex] $(u^*-\varphi)(\bar x) = \max_K(u^*-\varphi)$
\end{minipage}
\right\}
\quad\Longrightarrow\quad\text{$F_*(\nabla\varphi(\bar x),\nabla^2\varphi(\bar x))\le1$,}
\]
where an upper (lower) star denotes upper (lower) semicontinuous envelope (restricting the function to $K$). We say that $u$ has \emph{zero boundary condition} (in the viscosity sense) if
\[
\left.
\begin{minipage}[c][3em]{.35\textwidth}\center
$(\bar x,\varphi)\in \partial K \times C^2(\R^d)$ and \\[1ex] $(u^*-\varphi)(\bar x) = \max_K(u^*-\varphi)$
\end{minipage}
\right\}
\quad\Longrightarrow\quad\text{$F_*(\nabla\varphi(\bar x),\nabla^2\varphi(\bar x))\le1$ or $u^*(\bar x)\le 0$.}
\]
The function $u$ is said to be a {\em viscosity supersolution} in $\interior(K)$ with \emph{zero boundary condition} if the same conditions hold with $u^*$, $F_*$, $\max$, $\le$ replaced by $u_*$, $F^*$, $\min$, $\ge$. It is a {\em viscosity solution}  in $\interior(K)$ with \emph{zero boundary condition} if it is both a viscosity sub- and supersolution in $\interior(K)$ with zero boundary condition.

 To prove Theorem~\ref{T_usc_vs}, we must establish the sub- and supersolution properties. We carry out these tasks separately in the following two subsections. To do so, the following description of the semicontinuous envelopes of $F$ will be needed.

\begin{lemma} \label{L:191122}
The nonlinearity \eqref{MCF_FpM} satisfies $F_*=F$, as well as $F^*(p,M)=F(p,M)$ for $p\ne0$, and $F^*(0,M)=-\lambda_2(M)/2$. Here $\lambda_1(M)\ge\lambda_2(M)\ge\cdots\ge\lambda_d(M)$ are the eigenvalues of $M\in\S^d$. In particular, $F$ is continuous on the set $(\R^d\setminus\{0\})\times\S^d$.
\end{lemma}

\begin{proof}
From the representation \eqref{eq_infyF} we have $F(p,M)=-\frac12\sup\{y^\top M y\colon \text{$|y|=1$, $y^\top p=0$}\}$. One checks that this is continuous on the set $(\R^d\setminus\{0\})\times\S^d$, and in particular equal to $F_*$ and $F^*$ there. Next, we claim that
\begin{equation}\label{eq_Fpm_repr0001}
-\frac12\lambda_1(M) \le F(p,M) \le -\frac12\lambda_2(M)
\end{equation}
for all $(p,M)$. The first inequality follows because $\sup\{y^\top M y\colon |y|=1\}=\lambda_1(M)$. For the second inequality, use the spectral theorem to write $M=\lambda_1(M)w_1w_1^\top+\cdots+\lambda_d(M)w_dw_d^\top$ for an orthonormal basis $w_1,\ldots,w_d$ of eigenvectors of $M$. Express $p$ and $y$ is this basis, say $p=\pi_1w_1+\cdots+\pi_dw_d$ and $y=\eta_1w_1+\cdots+\eta_dw_d$, to get
\[
F(p,M) = -\frac12 \sup\left\{\sum_{i=1}^d \eta_i^2 \lambda_i(M) \colon \text{$\sum_{i=1}^d \eta_i^2=1$, $\sum_{i=1}^d \eta_i \pi_i =0$}\right\}.
\]
If $\pi_1=0$ one can take $\eta_1=1$ and $\eta_i=0$ for $i\ge2$ to get $F(p,M)\leq -\lambda_1(M)/2$. Otherwise one can take $\eta_2=(1+(\pi_2/\pi_1)^2)^{-1/2}$ and $\eta_1=-\eta_2 \pi_2/\pi_1$ to get $F(p,M)\le-(\eta_1^2\lambda_1(M)+\eta_2^2\lambda_2(M))/2\le-\lambda_2(M)/2$. In either case, the second inequality of \eqref{eq_Fpm_repr0001} holds.

For any fixed $M$, there is a sequence $(p_n,M_n)\to(0,M)$ with $F_*(0,M)=\lim_n F(p_n,M_n)$. Thus by \eqref{eq_Fpm_repr0001} and since $\lambda_1(M)$ is continuous in $M$, we get
\[
F(0,M) \ge F_*(0,M)=\lim_n F(p_n,M_n) \ge -\frac12 \lim_n\lambda_1(M_n)=-\frac12\lambda_1(M)=F(0,M).
\]
This shows that $F_*(0,M)=F(0,M)$. On the other hand, with $w_1$ an eigenvector of $M$ with eigenvalue $\lambda_1(M)$, we have $F(n^{-1}w_1,M)=-\lambda_2(M)/2$. Sending $n\to\infty$ shows that $F^*(0,M)\ge -\lambda_2(M)/2$ and thus, by \eqref{eq_Fpm_repr0001} and the continuity of  $\lambda_2(M)$  in $M$, that $F^*(0,M)= -\lambda_2(M)/2$.
\end{proof}

For later use, let us also record the following observations. We let $|\cdot|_\text{op}$ denote the operator norm of a matrix.

\begin{lemma} \label{L:211205.2}
	If $p \in \R^d$, $M \in \S^d$, $F^*(p, M) > 0$, and $B$ is an $d \times d$ invertible matrix then 
	\[
		F^*(p,M) \leq | (BB^\top)^{-1}|_\text{\textnormal{op}}\, F^*(B^\top p, B^\top M B).
	\]
\end{lemma}

\begin{proof}
Assume first $p \neq 0$, so that $F^*(p, M) = F(p, M)$ by Lemma~\ref{L:191122}. 
Consider any $\bar a\in\S^d_+$ with $\bar a B^\top p=0$ and $\tr(\bar a)=1$. 
Then $\tr(B \bar a B^\top) > 0$. 
Define now $ a=(B \bar a B^\top) / \tr(B \bar a B^\top)$. Then $a\in\S^d_+$, $a p=0$, and $\tr(a)=1$. Thus from the definition of $F$,
\begin{align} \label{eq:211205.1}
F^*(p, M) = F(p, M) \leq -\frac{1}{2} \tr(aM) = -\frac{1}{2} \tr(\bar a B^\top M B) \frac{1}{\tr(B \bar a B^\top)}.
\end{align}
Since $F^*(p, M) \geq 0$ we have $\tr(\bar a B^\top M B) \leq 0$. Moreover, using the spectral theorem we obtain $1= \tr(B \bar a B^\top  (BB^\top)^{-1}) \le \tr(B \bar a B^\top) | (BB^\top)^{-1}|_\text{op}$. Thus \eqref{eq:211205.1} becomes
\[
F^*(p, M) \le -\frac{1}{2} \tr(\bar a B^\top M B) | (BB^\top)^{-1}|_\text{op},
\]
and taking infimum on the right-hand side gives the assertion, still for $p \ne 0$.

Consider now the case $p=0$ and consider a sequence $(p_n, M_n)$ converging to $(p, M)$ with $p_n \neq 0$ such that $\lim_n F(p_n, M_n) = F^*(p, M)$. Since for sufficiently large $n$ we have $F(p_n, M_n)  > 0$, we get from the case just established that
\begin{align*}
	 F^*(p, M) &= \lim_n F(p_n, M_n)  \leq  | (BB^\top)^{-1}|_\text{op} \limsup_n 
	 	F^*(B^\top p_n, B^\top M_n B) \\
		&\le  | (BB^\top)^{-1}|_\text{op} F^*(B^\top p, B^\top M B),
\end{align*}
as desired.
\end{proof}

\begin{corollary}\label{C:211205}
	Let $B$ be a $d \times d$ invertible matrix, viewed as a linear map. 
	Define $K' = B^{-1}(K)$ and $\bar w = | (BB^\top)^{-1}|_\text{\textnormal{op}}\, w \circ B$. If, on $K$, $w$ is a lower semicontinuous viscosity supersolution of \eqref{PDE_1} with zero boundary condition, then so is $\bar w$ on $K'$. 
\end{corollary}
\begin{proof}
	The statement follows from the definition of viscosity supersolution, in conjunction with Lemma~\ref{L:211205.2}.
\end{proof}

\subsection{Subsolution property}

We now prove the subsolution property claimed in Theorem~\ref{T_usc_vs}. Since $v$ is upper semicontinuous and $F$ is lower semicontinuous, we may drop the stars in the definition of subsolution.

\begin{proof}[Proof of the subsolution property]
Fix $\bar x\in K$. If $\bar x \in \interior(K)$ then $v(\bar x) > 0$ by Example~\ref{Ex:190511}. If $\bar x \in \partial K$ and $v(\bar x) = 0$ then the subsolution property holds for this point.  Hence, without loss of generality, we may assume that $v(\bar x) > 0$. 

Fix now $\varphi\in C^2(\R^d)$ with $\varphi(\bar x)=v(\bar x)$ and $\varphi(x) \geq v(x)$ for all $x\ne\bar x$. We assume that $F(\nabla\varphi(\bar x),\nabla^2\varphi(\bar x))>1$ and work towards a contradiction. Without loss of generality, we may assume $\varphi(x) > v(x)$ for all $x\ne\bar x$.

We claim that there exists  $\varepsilon  \in (0, \sqrt{v(\bar x)}/2)$ such that
\begin{equation}\label{T_MCF_char_pf2_new}
\begin{split}
&\text{for all $(x,a)\in (K \cap B_\varepsilon(\bar x))\times\S^d_+$ with $\tr(a)=1$, we have}\\
&\ 1+\frac12\tr(a\nabla^2\varphi(x)) >0  \quad \text{implies} \quad   \nabla\varphi(x)^\top a \nabla\varphi(x) \geq \varepsilon.
\end{split}
\end{equation}
Indeed, if not, there exist $\varepsilon_n\to0$ and $(x_n,a_n)\in (K \cap B_{\varepsilon_n}(\bar x)) \times\S^d_+$ such that $\tr(a_n)=1$ and  $\nabla\varphi(x_n)^\top a_n \nabla\varphi(x_n)\le\varepsilon_n$, but $1+\frac12\tr(a_n \nabla^2\varphi(x_n)) > 0$. In particular, $x_n\to\bar x$ and, after passing to a subsequence, we also have $a_n\to a$ for some $a\in\S^d_+$. Passing to the limit yields $\tr(a)=1$, $a\nabla\varphi(\bar x)=0$, and $1+\frac12\tr(a \nabla^2\varphi(\bar x))\ge0$. This contradicts the assumption that $1-F(\nabla\varphi(\bar x),\nabla^2\varphi(\bar x))<0$, and proves the claim. 

Note also that there exists some $c>0$ such that for
all $(x,a)\in (K \cap B_\varepsilon(\bar x)) \times\S^d_+$ with $\tr(a)=1$ we have
\begin{equation}\label{eq:191113.1}
1+\frac12\tr(a \nabla^2\varphi(x)) \leq  1 + \frac{1}{2} \lambda_1(\nabla^2\varphi( x)) \leq c < \infty,
\end{equation}
where $\lambda_1(M)$ denotes the largest eigenvalue of a symmetric matrix $M$. The boundedness comes from the continuity of $\lambda_1$.
Furthermore, we have
\begin{equation}\label{T_MCF_char_pf3_new}
\delta = \min_{K \cap \partial B_\varepsilon(\bar x)}(\varphi - v) > 0.
\end{equation}

Fix any optimal $\P\in\Pcal_{\bar x}$.  We then have a predictable $\S^d_+$-valued process  $(a(s))_{s\ge0}$  such that
\begin{equation*}
\langle X\rangle(t) = \int_0^t a(s) ds \quad\text{and}\quad \text{$\tr(a(t))=1$, $dt\otimes d\P$-a.e.}
\end{equation*}
Define the stopping time
\[
\theta = \inf\{t\ge0\colon X(t) \notin B_\varepsilon(\bar x)\} \wedge v(\bar x).
\]
Clearly $\theta\le\tau_K$ by definition of $v(\bar x)$ and $\P[X(\theta) \in K \cap \partial  B_\varepsilon(\bar x)] > 0$ since $\varepsilon < \sqrt{v(\bar x)}/2$ (recall Example~\ref{Ex:190511}).

We can now  define the predictable set
\[
J = \{s\in[0,\theta)\colon 1+\frac12\tr(a(s)\nabla^2\varphi(X(s))) > 0\}.
\]
Next, the dynamic programming principle of Proposition~\ref{P_v_proporties}\ref{P_v_proporties_3} yields
\begin{equation}\label{T_MCF_char_pf01_new}
v(\bar x) \le t \wedge\theta  + v(X(t \wedge \theta  )), \quad \text{$\P$--a.s.}
\end{equation}
Using \eqref{T_MCF_char_pf01_new}  and then \eqref{T_MCF_char_pf3_new}, we get
\[
\varphi(\bar x) = v(\bar x) \le t\wedge\theta + v(X(t\wedge\theta)) \le t\wedge\theta - \delta\bm1_{[\theta,\infty)}(t)   \bm1_{\{X(\theta) \in K \cap  \partial  B_\varepsilon(\bar x) \}}  + \varphi(X(t\wedge\theta)).
\]
Combining this with It\^o's formula, the definition of $J$, and \eqref{eq:191113.1}, we get
\[
\begin{aligned}
\delta \bm1_{[\theta,\infty)}(t) \bm1_{\{X(\theta) \in K \cap  \partial  B_\varepsilon(\bar x) \}} &\le t\wedge\theta +\varphi(X(t\wedge\theta))-\varphi(\bar x) \\
&= \int_0^{t\wedge\theta} \nabla\varphi(X(s))^\top dX(s) + \int_0^{t\wedge\theta} (1+\frac12\tr(a(s)\nabla^2\varphi(X(s))))ds \\
&\le \int_0^{t\wedge\theta} \nabla\varphi(X(s))^\top dX(s) + c \int_0^{t\wedge\theta} \bm1_J(s) ds.
\end{aligned}
\]
Now, define the process
\begin{equation*}
\widetilde X(t) = X(t) + \frac{c}{\varepsilon}\int_0^t a(s)\nabla\varphi(X(s)) \bm1_J(s)ds.
\end{equation*}
Due to \eqref{T_MCF_char_pf2_new} and the definition of~$J$, we then have
\begin{align}
\delta \bm1_{[\theta,\infty)}(t) \bm1_{\{X(\theta) \in K \cap  \partial  B_\varepsilon(\bar x) \}} &\le  \int_0^{t\wedge\theta} \nabla\varphi(X(s))^\top d\widetilde X(s) \nonumber\\
&\qquad+ \int_0^{t\wedge\theta} (c - \frac{c}{\varepsilon} \nabla\varphi(X(s))^\top a(X(s)) \nabla\varphi(X(s)) )\bm1_J(s)ds \nonumber\\
&\le \int_0^{t\wedge\theta} \nabla\varphi(X(s))^\top d\widetilde X(s).   \label{T_MCF_char_pf013_new}
\end{align}
Consider now the exponential local martingale $Z$ given by
\[
\frac{dZ(t)}{Z(t)} = -\frac{c}{\varepsilon}\bm1_J(t)\nabla\varphi(X(t))^\top dX(t), \quad Z_0=1.
\]
This is well-defined since $\nabla\varphi$ is bounded on the closure of $B_\varepsilon(\bar x)$, which contains $X(t)$ for $t\in J$. An application of It\^o's formula shows that multiplying \eqref{T_MCF_char_pf013_new} by $Z(t)$ gives a local martingale, and hence a supermartingale since it is nonnegative. Therefore,
\[
0 <  \delta\, \E[\bm1_{\{X(\theta) \in K \cap  \partial  B_\varepsilon(\bar x) \}} Z(\theta)] \le \E\left[ Z(\theta) \int_0^\theta \nabla\varphi(X(s))^\top d\widetilde X(s)\right] \le 0,
\]
using that $\theta<\infty$, $\P$-a.s., and $\P[X(\theta) \in K \cap \partial  B_\varepsilon(\bar x)] > 0$ for the first inequality. This contradiction completes the proof of the subsolution property.
\end{proof}

\subsection{Supersolution property}

The following result is used in the proof.

\begin{lemma} \label{L_weaksolution}
Let $m\in\N$ with $m\ge2$, and let $S$ be a nonzero skew-symmetric $m\times m$ matrix and let $x,\bar x\in\R^m$. Then there exists a weak solution to the SDE
\[
dY(t) = \frac{S(Y(t)-\bar x)}{|S(Y(t)-\bar x)|}dW(t), \quad Y(0)=x,
\]
that satisfies $|S(Y(t)-\bar x)|^2=|S(x-\bar x)|^2+t$ for all $t\ge0$. Here $W$ denotes a one-dimensional Brownian motion.
\end{lemma}

\begin{proof}
Suppose first that $S(x-\bar x)\ne0$. Since the SDE has locally Lipschitz coefficients on the set $\{y\colon S(y-\bar x)\ne0\}$, there is a local solution $Y$ on $[0,\zeta)$, where $\zeta=\inf\{t\ge0\colon S(Y(t)-\bar x)=0\}$. It\^o's formula and the skew-symmetry of $S$ give
\[
d|S(Y(t)-\bar x)|^2 = 2 \frac{(Y(t)-\bar x)^\top S^\top S^2(Y(t)-\bar x)}{|S(Y(t)-\bar x)|}dW(t) + dt = dt, \quad t<\zeta,
\]
so $\zeta=\infty$. Thus $Y$ is actually a global solution, and $|S(Y(t)-\bar x)|^2=|S(x-\bar x)|^2+t$ for all $t\ge0$. This proves the case where $S(x-\bar x)\ne0$.

Suppose now that $S(x-\bar x)=0$, and select points $x_n\in\R^m$ with $x_n\to x$ and $S(x_n-\bar x)\ne0$. For each $n$, let $Y_n$ be a solution to the SDE with $Y_n(0)=x_n$. Since $\tr\langle Y_n\rangle(t)\equiv t$, the law of $Y_n-x_n$ lies in $\Pcal_0$, which is compact by Proposition~\ref{P_v_proporties}\ref{P_v_proporties_1}. Thus after passing to a subsequence, we have $Y_n-x_n\Rightarrow Y-x$ for some limiting process $Y$ with $Y(0)= x$. Since the set $C=\{\omega\colon\text{$|S(\omega(t)-\bar x)|^2=|S(\omega(0)-\bar x)|^2+t$ for all $t\ge0$}\}$ is closed, and since $Y_n$ lies in $C$ almost surely for all $n$, the Portmanteau lemma implies that $Y$ does as well. In particular, we have $|S(Y(t)-\bar x)|^2=t$ for all $t\ge0$. Now, for every $f\in C^\infty_c(\R^m)$, $k\in\N$, $0\le s_1\le\cdots\le s_k< s< t$, and $g\in C_b((\R^m)^k)$, we have
\begin{equation}\label{eq_L_weaksolution_1}
\E\left[ \left( f(Y_n(t)) - f(Y_n(s)) - \int_s^t Lf(u,Y_n(u))du\right) g(Y_n(s_1),\ldots,Y_n(s_k))\right] = 0,
\end{equation}
where $Lf(u,y)=\frac12(y-\bar x)^\top S^\top \nabla^2 f(y) S(y-\bar x) / (|S(y-\bar x)|^2+u)$ is the operator associated to the given SDE. Note that this uses that $|S(Y_n(u)-\bar x)|^2=|S(x-\bar x)|^2+u$. The expression inside the expectation on the left-hand side of \eqref{eq_L_weaksolution_1} is a bounded continuous function of the trajectory of $Y_n$. We may therefore pass to the limit and deduce that the corresponding equality holds for $Y$ as well. It follows that $Y$ solves the martingale problem problem associated with the given SDE. Equivalently, $Y$ is a weak solution, as desired.
\end{proof}

We now turn to the supersolution property claimed in Theorem~\ref{T_usc_vs}.

\begin{proof}[Proof of the supersolution property]
Fix $\bar x\in K$. If $\bar x \in \partial K$ then there is nothing to prove since $v$ is nonnegative. Hence, we may assume throughout the proof that $\bar x \in \Ocal$, where we write $\Ocal = \interior(K)$.

Fix now $\varphi\in C^2(\R^d)$ with $\varphi \le v_*$ and $\varphi(\bar x)=v_*(\bar x)$. A standard perturbation argument relying on test functions $\varphi(x)-\varepsilon|x-\bar x|^2$ lets us suppose that $\varphi(x)<v_*(x)$ for all $x\ne\bar x$, and that the Hessian $\nabla^2\varphi(\bar x)$ is nonsingular. We consider three cases, depending on the properties of $\nabla\varphi(\bar x)$ and $\nabla^2\varphi(\bar x)$.

{\it Case~1:} Suppose $\nabla\varphi(\bar x)\ne0$. Assume for contradiction that $F^*(\nabla\varphi(\bar x),\nabla^2\varphi(\bar x))<1$. Since $F^*$ equals $F$ at this point, it follows that there exists $\bar\sigma\in\R^d$ such that
\[
\text{$|\bar\sigma|=1$, $\bar\sigma^\top\nabla\varphi(\bar x)=0$, and $1+\frac12\bar\sigma^\top\nabla^2\varphi(\bar x)\bar\sigma>0$.}
\]
In particular, there exists a skew-symmetric $d\times d$ matrix $S$ such that $\bar\sigma = S\nabla\varphi(\bar x)$; for instance,
$$S=\frac{1}{|\nabla \varphi(\bar x)|^2} (\bar\sigma \nabla\varphi(\bar x)^\top -  \nabla\varphi(\bar x) \bar \sigma^\top).$$
 Furthermore, we can select $\varepsilon>0$ such that the closure of $B_\varepsilon(\bar x)$ is contained in $\Ocal$ and
\begin{equation} \label{eq_supsol_pf1}
\text{$|S\nabla\varphi|\ge\frac{1}{2}$ and $|S\nabla\varphi|^2+\frac12\nabla\varphi^\top S^\top\nabla^2\varphi \,S\nabla\varphi\ge0$ on $B_\varepsilon(\bar x)$.}
\end{equation}
Fix any $x\in B_\varepsilon(\bar x)$. Define
\begin{equation} \label{eq_supsol_pf_theta}
\theta = \inf\{t\ge0\colon X(t) \notin B_\varepsilon(\bar x)\},
\end{equation}
and let $\P$ be the law under which $X$ satisfies
\begin{equation}\label{T_MCF_char_pf00003}
dX(t) = \left( \frac{S\nabla\varphi(X(t))}{|S\nabla\varphi(X(t))|}\bm1_{[0,\theta)}(t) + e_1\bm1_{[\theta,\infty)}(t)\right) dW(t), \quad X_0=x,
\end{equation}
where $W$ is a one-dimensional Brownian motion and $e_1$ is the first canonical unit vector (any other unit vector would also do). Note that $\P\in\Pcal_x$ and $\theta\le\tau_K$, and thus $\theta<\infty$, $\P$-a.s.\ by Lemma~\ref{L_E[theta]}. Define
\[
\delta = \min_{\partial B_\varepsilon(\bar x)}(v_*-\varphi) > 0.
\]
Using first that $v\ge v_*\ge\varphi+\delta$ on $\partial B_\varepsilon(\bar x)$; then It\^o's formula; and finally \eqref{eq_supsol_pf1} along with the fact that $\nabla\varphi^\top S\,\nabla\varphi=0$ by skew-symmetry of $S$, we get
\begin{align*}
\theta &+ v(X(\theta)) \ge \delta + \theta + \varphi(X(\theta)) \\
&= \delta + \varphi(x) + \int_0^\theta \frac{\nabla\varphi^\top S\,\nabla\varphi}{|S\nabla\varphi|}(X(s))dW(s) + \int_0^\theta ( 1 + \frac{\nabla\varphi^\top S^\top\nabla^2\varphi \,S\nabla\varphi}{2|S\nabla\varphi|^2}(X(s))) ds \\
&\ge \delta + \varphi(x), \quad \text{$\P$-a.s.}
\end{align*}
Combining this with the dynamic programming principle of Proposition~\ref{P_v_proporties}\ref{P_v_proporties_3} yields
\[
v(x) \ge \essi{\P} \{\theta + v(X(\theta))\} \ge \delta + \varphi(x).
\]
Since $x\in B_\varepsilon(\bar x)$ was arbitrary, we may send $x\to \bar x$ such that $v(x)\to v_*(\bar x)=\varphi(\bar x)$, and deduce $0\ge\delta$. This contradiction proves the supersolution property when $\nabla\varphi(\bar x)\ne0$.

{\it Case 2:} Suppose now that $\nabla\varphi(\bar x)=0$ and $\nabla^2\varphi(\bar x)$ is negative definite. Assume for contradiction that $F^*(0,\nabla^2\varphi(\bar x))<1$, meaning that $1+\lambda_2(\nabla^2\varphi(\bar x))/2>0$. We will replace $\varphi$ by a simpler test function $\widetilde\varphi$. To this end, define $\gamma_i=\lambda_i(\nabla^2\varphi(\bar x))-\eta$ for $i=1,\ldots,d$, where $\eta>0$ is small enough so that $1+\gamma_2/2\ge0$. Let $w_1,\ldots,w_d$ be an orthonormal basis of eigenvectors of $\nabla^2\varphi(\bar x)$ corresponding to its ordered eigenvalues. Define
\[
M = \gamma_2(w_1w_1^\top + w_2w_2^\top) + \gamma_3w_3w_3^\top + \cdots + \gamma_dw_dw_d^\top
\]
and
\[
\widetilde\varphi(x)=v_*(\bar x)+\frac12(x-\bar x)^\top M(x-\bar x).
\]
Then $\widetilde\varphi(\bar x)=\varphi(\bar x)=v_*(\bar x)$, $\nabla\widetilde\varphi(\bar x)=\nabla\varphi(\bar x)=0$, and $\nabla^2\widetilde\varphi(\bar x)=M\prec\nabla^2\varphi(\bar x)$. Thus $\widetilde\varphi \le\varphi$ on some ball $B_\varepsilon(\bar x)$ with positive radius $\varepsilon>0$, whose closure is contained in $\Ocal$. Define the skew-symmetric matrix
\[
S = w_1w_2^\top - w_2w_1^\top.
\]
Fix any $x\in B_\varepsilon(\bar x)$, and let $\P$ be a law under which $X$ satisfies
\[
dX(t) = \frac{S(X(t)-\bar x)}{|S(X(t)-\bar x)|} dW(t), \quad X(0)=x,
\]
and $|S(X(t)-\bar x)|^2=|S(x-\bar x)|^2+t$, where $W$ is a one-dimensional Brownian motion. Such $\P$ exists by Lemma~\ref{L_weaksolution}, and it is clear that $\P\in\Pcal_x$. It\^o's formula, the identity $MS=\gamma_2 S$, and the skew-symmetry of $S$ give
\begin{align*}
\widetilde\varphi(X_t) &= \widetilde\varphi(x) + \int_0^t \frac{(X(s)-\bar x)^\top MS(X(s)-\bar x)}{|S(X(s)-\bar x)|} dW(s) \\
&\quad + \frac12 \int_0^t \frac{(X(s)-\bar x)^\top S^\top MS(X(s)-\bar x)}{|S(X(s)-\bar x)|^2}ds \\
&=\widetilde\varphi(x) + \frac{\gamma_2}{2}t, \quad t\ge0,\quad \text{$\P$-a.s.}
\end{align*}
As in Case~1, let $\theta$ be given by \eqref{eq_supsol_pf_theta} and define $\delta = \min_{\partial B_\varepsilon(\bar x)}(v_*-\widetilde\varphi)>0$. We then get
\[
\theta+v(X(\theta)) \ge \delta+\theta+\widetilde\varphi(X(\theta)) \ge \delta+ \widetilde\varphi(x) + \left(1+\frac{\gamma_2}{2}\right)\theta \ge \delta +\widetilde\varphi(x), \quad \text{$\P$-a.s.},
\]
using that $1+\gamma_2/2\ge0$. The contradiction $v_*(\bar x)\ge\delta+v_*(\bar x)$ is now obtained as in Case~1 using the dynamic programming principle and a limiting argument.

{\it Case~3:} Suppose finally that $\nabla\varphi(\bar x)=0$ and $\nabla^2\varphi(\bar x)$ has at least one strictly positive eigenvalue with eigenvector $\hat e$, say. Fix $\varepsilon_0>0$ such that the closure of $B_{\varepsilon_0}(\bar x)$ is contained in $\Ocal$, and define
\begin{equation}\label{T_MCF_char_pf000001}
\delta = \min_{\partial B_{\varepsilon_0}(\bar x)}(v_*-\varphi) > 0.
\end{equation}
Following \cite{son_tou_02} (specifically, Steps 6--7 in the proof of Theorem~4.1, see Section~8.2 in their paper), we define perturbed test functions
\[
\varphi_\varepsilon(x) = \varphi(x) + \varepsilon \hat e^\top (x-\bar x).
\]
The minimum of $v_*-\varphi_\varepsilon$ over the closure of $B_{\varepsilon_0}(\bar x)$ is at most $v_*(\bar x)-\varphi_\varepsilon(\bar x)=0$. Because of \eqref{T_MCF_char_pf000001}, for every sufficiently small $\varepsilon>0$, the minimum cannot be attained on the boundary $\partial B_{\varepsilon_0}(\bar x)$, so must be attained at some $x_\varepsilon\in B_{\varepsilon_0}(\bar x)$. Moreover, since $\bar x$ is a strict minimizer of $v_*-\varphi$, we have $x_\varepsilon\to\bar x$ as $\varepsilon\to0$. The argument in Step~7 of the proof of Theorem~4.1 in \cite{son_tou_02}, which makes use of the fact that $\hat e$ is an eigenvector with strictly positive eigenvalue, yields that
\[
\text{$\nabla\varphi_\varepsilon(x_\varepsilon) \ne 0$ for all sufficiently small $\varepsilon$.}
\]
Therefore, the result proved in Case~1 above implies that $F^*(\nabla\varphi_\varepsilon(x_\varepsilon),\nabla^2\varphi_\varepsilon(x_\varepsilon))\ge1$. Sending $\varepsilon\to0$ gives $F^*(\nabla\varphi(\bar x),\nabla^2\varphi(\bar x))\ge1$, which completes the proof of the supersolution property.
\end{proof}

\section{Comparison and uniqueness}\label{S_comparison}

The main result of this section is the following comparison principle, which is used to prove Theorem~\ref{T_uniqueness}.

\begin{theorem}\label{T_comparison}
Let $d\ge2$ and suppose $K$ is compact. Assume there exist invertible affine maps $T_\lambda$ on $\R^d$, parameterized by $\lambda\in(0,1)$, such that $T_\lambda(K)\subset\interior(K)$ and $\lim_{\lambda\to1}T_\lambda=I$. Let $u$ ($w$) be an upper (lower) semicontinuous viscosity subsolution (supersolution) of \eqref{PDE_1}, both $u$ and $w$ with zero boundary condition (in the viscosity sense). Then $u\le w^*$.
\end{theorem}

Before giving the proof, let us show how this implies Theorem~\ref{T_uniqueness}. Let $u$ and $v$ be two upper semicontinuous viscosity solutions of \eqref{PDE_1} with zero boundary condition. Applying the comparison principle with $w=v_*$ yields $u\le(v_*)^*\le v^*=v$. Letting $u$ and $v$ switch places yields $v\le u$, and hence $u=v$.

\begin{example}\label{E_T_lambda}
If $K$ is strictly star-shaped about the origin, meaning that $\lambda K\subset\interior(K)$ for all $\lambda\in(0,1)$, then it clearly satisfies the assumption of Theorem~\ref{T_comparison}. In particular, this is the case if $K$ is convex with $0\in\interior(K)$. Here is an example of a body that is not star-shaped but satisfies the assumptions of Theorem~\ref{T_comparison}:
\[
K = [-1,1]^2 \cap \{(x,y)\colon |y|\le 0.01 + x^2\}.
\]
Indeed, one can use the linear maps $T_\lambda(x,y)=(\lambda x, \lambda^2y)$. It is easily verified that $K$ is not star-shaped.
\end{example}

If $K$ is star-shaped but not strictly star-shaped, then uniqueness among upper semicontinuous viscosity solutions may fail, as the following example shows; see also \citet[Example~8.2]{Soner:1993}.

\begin{example}\label{E_non_strict_star}
Let $D$ be the centered unit disk in $\R^2$, and set $K=(D+(1,0))\cup(D-(1,0))$. Then $K$ is star-shaped because $\lambda K\subset K$ for all $\lambda\in(0,1)$, but not strictly star-shaped because $\interior(K)$ is not connected. The value function is upper semicontinuous and satisfies $v(0,0)\ge1$, which can be seen by the argument in the proof of Proposition~\ref{P_quasi_concave} below. However, it is easy to verify that the function $\bar v(x,y)=1-(|x|-1)^2-y^2$ for $(x,y)\in K$ is a (continuous) viscosity solution of \eqref{PDE_1} with zero boundary condition. Since $\bar v(0,0)=0\ne v(0,0)$, this shows non-uniqueness.
\end{example}

The proof of Theorem~\ref{T_comparison} relies on the following maximum principle, which holds for arbitrary compact sets $K$. The boundary conditions in its statement should be understood in the viscosity sense.

\begin{theorem} \label{T:211205}
Let $d\ge2$ and suppose $K$ is compact. Let $u$ ($w$) be an upper (lower) semicontinuous viscosity subsolution (supersolution) of \eqref{PDE_1}.  Then there exists $\bar x \in \partial K$ that achieves $\max_K (u-w)$.  Moreover, if $u$ in addition satisfies the zero boundary condition, and if $w$ is  a lower semicontinuous viscosity supersolution of \eqref{PDE_1} with zero boundary condition on some compact set $K'$ such that $K \subset \interior(K')$, then $u \leq w$ on $K$. 
\end{theorem}
\begin{proof}
We proceed in several steps.

1. It is enough to prove the two assertions with $u$ replaced by $\delta u$, for each $\delta \in (0,1)$. Indeed, if $\delta u \leq w$ for all $\delta \in (0,1)$ and if $w$ is nonnegative (see Lemma~\ref{L:211205} below) then also $u \leq w$, yielding the second assertion. For the first assertion, assume we have $\bar x_\delta \in \partial K$ that achieves $\max_K(\delta u-w)$, for each $\delta \in (0,1)$. Then there exists a sequence $(\delta_n)$ such that $\lim_n \delta_n = 1$ and $\lim_n \bar x_{\delta_n} = \bar x$ for some $\bar x \in \partial K$. 
Then for all $x \in K$ we have
\begin{align*}
	\delta_n u(x) - w(x) \leq \delta_n u\left(\bar x_{\delta_n}\right) - w\left(\bar x_{\delta_n}\right) = u\left(\bar x_{\delta_n}\right) - w\left(\bar x_{\delta_n}\right) - (1 - \delta_n) u\left(\bar x_{\delta_n}\right).
\end{align*}
Sending $n$ to infinity and using upper semicontinuity of $u-w$ then shows that $\bar x$ achieves $\max_K (u-w)$ as required.
Now, $\delta u$ is a subsolution of the equation
\begin{equation}\label{eq_comp_1}
F(\nabla u,\nabla^2 u) = \delta,
\end{equation}
and if $u$ satisfies the zero boundary condition, then so does $\delta u$.
Thus, by writing $u$ instead of $\delta u$, we may and do assume throughout the proof that $u$ itself is a subsolution of \eqref{eq_comp_1}, where $\delta\in(0,1)$ is arbitrary but fixed. 
 
2. For the first assertion, for every $\varepsilon>0$, define
\[
\Phi_\varepsilon(x,y) = u(x) - w(y) - \frac{1}{\varepsilon^4}|x-y|^4
\]
for $(x,y)\in K\times K$, and let $(x_\varepsilon,y_\varepsilon)$ maximize $\Phi_\varepsilon$ over $K\times K$. Then we have
\begin{align} \label{eq:211204.1}
\Phi_\varepsilon(x_\varepsilon,y_\varepsilon) \ge \max_{x\in K} \Phi_\varepsilon(x, x) = \max_K (u - w).
\end{align}
By compactness, $(x_\varepsilon,y_\varepsilon)$ converges to some $(\bar x,\bar y)\in K\times K$ as $\varepsilon\to0$ along a suitable subsequence; in the following, $\varepsilon$ is always understood to belong to this subsequence. Since $\varepsilon^{-4}|x_\varepsilon-y_\varepsilon|^4 \le \max_K u - \min_K w$, we actually have $\bar x= \bar y$. Moreover, \eqref{eq:211204.1} yields
\[
	\max_K(u-w) \leq \limsup_{\varepsilon \to 0} \Phi_\varepsilon(x_\varepsilon,y_\varepsilon) \leq 
	\limsup_{\varepsilon \to 0}  \left(u(x_\varepsilon) - w(y_\varepsilon) \right)
	\leq u(\bar x) - w (\bar x),
\]
by upper semicontinuity of $u$ and of $-w$. Hence
$\bar x$ maximizes $u-w$ over $K$.
Thus, to show the first assertion it suffices to argue that $(x_\varepsilon, y_\varepsilon) \in \interior(K) \times \interior(K)$ is not possible. This forces $\bar x \in \partial K$ as desired.

For the second assertion, we define $\Phi_\varepsilon$ as above, but now on the larger set $K \times K'$.   Let $(x_\varepsilon,y_\varepsilon)$ again denote the corresponding maximizers, which converge along a subsequence to some $(\bar x, \bar y) \in  K \times K'$.  We again obtain $\bar x = \bar y$, thus $\bar y \in K \subset \interior(K')$, and therefore $y_\varepsilon \in \interior(K')$ for all sufficiently small $\varepsilon$. We will use this to argue that $u$ cannot satisfy the viscosity inequality at $x_\varepsilon$. This forces $x_\varepsilon \in \partial K$ and, due to the boundary condition, $u(x_\varepsilon) \leq 0$. Together with nonnegativity of $w$ (see Lemma~\ref{L:211205} below) this yields $\Phi_\varepsilon(x_\varepsilon,y_\varepsilon) \le 0$ and thus, thanks to \eqref{eq:211204.1}, $\max_K (u-w) \le 0$. This is the second assertion.

3. Both assertions can now be argued by contradiction in the same manner: for any fixed small $\varepsilon > 0$, we assume that both $u$ and $w$ simultaneously satisfy the viscosity inequalities at $x_\varepsilon$ and $y_\varepsilon$, respectively, and use this to derive a contradiction. (Indeed, to prove the first assertion we had to exclude that $(x_\varepsilon, y_\varepsilon) \in \interior(K) \times \interior(K)$, while for the second assertion we had to exclude that $y_\varepsilon \in \interior(K)$ and that $u$ satisfies the viscosity inequality at $x_\varepsilon$.)

4. Let us work under the assumptions of Step 3. Define
\[
\zeta(x,y) = \frac{1}{\varepsilon^4}|x-y|^4.
\]
To simplify notation, write
\begin{align*}
p &= \nabla_x \zeta(x_\varepsilon,y_\varepsilon) = 2 \varepsilon^{-4} |x_\varepsilon-y_\varepsilon|^2(x_\varepsilon-y_\varepsilon), \\
H &= \nabla^2_{xx}\zeta(x_\varepsilon,y_\varepsilon) = 2\varepsilon^{-4}|x_\varepsilon-y_\varepsilon|^2 I + 4 \varepsilon^{-4} (x_\varepsilon-y_\varepsilon) (x_\varepsilon-y_\varepsilon)^\top.
\end{align*}
Then $\nabla_y\zeta(x_\varepsilon,y_\varepsilon)=-p$, $\nabla^2_{xy}\zeta(x_\varepsilon,y_\varepsilon) = -H$, and $\nabla^2_{yy}\zeta(x_\varepsilon,y_\varepsilon) = H$. We also define
\begin{equation}\label{eq_comp_68}
A = \nabla^2\zeta(x_\varepsilon,y_\varepsilon) = \begin{pmatrix} H & -H \\ -H & H \end{pmatrix}.
\end{equation}

We now claim that $p\ne0$. Suppose for contradiction that $p=0$. Then $x_\varepsilon=y_\varepsilon$, $\nabla_y\zeta(x_\varepsilon,y_\varepsilon)=0$, and $\nabla_{yy}\zeta(x_\varepsilon,y_\varepsilon)=0$. Since $y_\varepsilon$ minimizes $y\mapsto w(y)+\zeta(x_\varepsilon,y)$ over $K$ (respectively, over $K'$), the supersolution inequality states that $0=F^*(0,0)\ge1$. This contradiction confirms that $p\ne0$.

 Ishii's lemma, see \citet[Theorem~3.2]{cra_ish_lio_92}, now gives $M,N\in\S^d$ such that
\[
(p,M)\in \overline J_K^{2,+}u(x_\varepsilon), \qquad (-p,N)\in \overline J_K^{2,-}w(y_\varepsilon),
\]
and
\begin{equation}\label{eq_comp_61}
\begin{pmatrix}M&0\\0&-N\end{pmatrix} \preceq A+A^2.
\end{equation}
Pre- and post-multiplying \eqref{eq_comp_61} by vectors of the form $(z,z)$ and using \eqref{eq_comp_68} shows that $M \preceq N$.  Now we use the fact that
 $(p,M)$ lies in limiting superjet of the subsolution $u$ at $x_\varepsilon$, the
ellipticity of $F$,  Lemma~\ref{L:191122}, the fact that $p\ne0$, and finally that $(-p,N)$ lies in the limiting subjet of the supersolution $w$ at $y_\varepsilon$ to get
\begin{equation*}
\delta \geq F(p,M)  \ge F(p,N) = F(-p,N) =F^*(-p,N)\ge1.
\end{equation*}
This is the required contradiction, which concludes the proof.
\end{proof}

We used the following observation in the previous proof. The boundary condition in its statement should be understood in the viscosity sense.

\begin{lemma} \label{L:211205}
	If $w$ is  a lower semicontinuous viscosity supersolution of \eqref{PDE_1} with zero boundary condition on some compact  $K \subset \R^d$ with $d \geq 0$  then $w \geq 0$. 
\end{lemma}
\begin{proof}
The constant test function $\varphi\equiv \min_K w$ certifies that $w\ge0$. Indeed, if $\bar x$ minimizes $w$ over $K$ and $w(\bar x)<0$, then the supersolution inequality holds regardless of whether $\bar x$ lies in the interior or on the boundary. Thus $0=F^*(\nabla\varphi(\bar x),\nabla^2\varphi(\bar x))=F^*(0,0)\ge1$, a contradiction. So $w(\bar x)\ge0$.
\end{proof}

We now give the proof of the comparison principle; see also  \citet[Section~9]{Soner:1993},  \citet[Section~2 and Theorem~4.3]{Barles:1993}, and \citet[Theorem~4]{koh_ser_06} for related uniqueness statements.

\begin{proof}[Proof of Theorem~\ref{T_comparison}]
We assume for simplicity that the $T_\lambda$ are linear, not just affine; we may then identify $T_\lambda$ with its $d \times d$ matrix. Recall that $| \cdot |_\text{\textnormal{op}}$ denotes the operator norm. By Corollary~\ref{C:211205}, the function $w_\lambda = | (T_\lambda T_\lambda^\top)^{-1}|_\text{\textnormal{op}}\, w \circ T_\lambda$ is a lower semicontinuous viscosity supersolution of \eqref{PDE_1} with zero boundary condition on  $K' = T_\lambda^{-1}(K)$. By the properties of $T_\lambda$, we have $K \subset \interior(K')$.   Theorem~\ref{T:211205} then yields $u \leq w_\lambda$ on $K$ for all $\lambda \in (0,1)$. We thus obtain $u \leq \limsup_{\lambda \to 1} w_\lambda \leq w^*$ on $K$ as desired.
\end{proof}

\section{Convex bodies}\label{S_convex}

Our next goal is to prove continuity of the value function $v$ when $K\subset\R^d$ ($d\ge2$) is a convex body satisfying an additional assumption. We first record the following simple property of the value function.

\begin{proposition}\label{P_quasi_concave}
Let $K$ be a convex body. Then $v$ is quasi-concave.
\end{proposition}

\begin{proof}
We must prove that $v$ has convex super-level sets. Pick two distinct points $x,y\in K$, and let $L$ be the line passing through $x$ and $y$. Fix any point $z\in L$, and let $\P$ be the law under which $X$ is a standard Brownian motion along $L$ starting at $z$. Then $\P\in\Pcal_z$, and with $\theta=\inf\{t\ge0\colon X(t)\in\{x,y\}\}$ the dynamic programming principle yields $v(z) \ge \essi{\P} \{\theta + v(X(\theta))\}\ge v(x)\wedge v(y)$. This proves quasi-concavity.
\end{proof}

Recall the following notions from convex geometry; see \citet{roc_70,sch_14} for more details. Let $F$ be any subset of $\R^d$. The affine span of $F$ is denoted by $\aff(F)$, with dimension $\dim(F)$. The relative interior $\ri(F)$ is the interior of $F$ in $\aff(F)$, and the relative boundary is $\rbd(F)=F\setminus\ri(F)$. A face of a convex set $K$ is a convex subset $F\subset K$ such that every (closed) line segment $L\subset K$ with $\ri(L)\cap F\ne\emptyset$ satisfies $L\subset F$. A face is called a boundary face if it is nonempty and not all of $K$. The relative boundary $\rbd(K)$ is the union of all boundary faces. For every $x\in K$, there is a unique face of $K$ whose relative interior contains $x$. We call this face $F_x$. For each $k=0,\ldots,d$, the $k$-skeleton is defined as in \eqref{F1Fd}, namely
\[
\text{$\Fcal_k = $ union of all faces $F$ of $K$ with $\dim(F)\le k$.}
\]
In particular, $\Fcal_0$ consists of all extreme points, $\Fcal_1$ consists of all extreme points and line segments, $\Fcal_{d-1}$ is the boundary of $K$, and $\Fcal_d$ is $K$ itself. For convenience we introduce the notation
\begin{equation} \label{eq: tauF}
\tau_F = \inf\{t\ge0\colon X(t) \notin F\}
\end{equation}
for the first exit time of $X$ from a set $F$. This notation is consistent with \eqref{eq: tauK}.

\begin{lemma}\label{L_FtauK}
Let $K$ be a convex body and consider a point $\bar x\in K$. For every $\P\in\Pcal_{\bar x}$, we have $\tau_K = \tau_{F_{\bar x}}$, $\P$-a.s.
\end{lemma}

\begin{proof}
If ${\bar x} \in \Fcal_d \setminus  \Fcal_{d-1}$, then $F_{\bar x} = K$, and the statement is obvious.  Otherwise, there exists a supporting halfspace $H_1 =\{x\in\R^d\colon a_1^\top x\ge b_1\}$ with  $(a_1,b_1)\in\R^d\times\R$ such that $K \subset H_1$ and  $\bar x \in \partial H_1$. Set $K_1 = K \cap \partial  H_1$ and note that $\dim(K_1) < \dim(K) = d$ and $F_{\bar x} \subset K_1$. The scalar process $a_1^\top X^{\tau_K}-b_1$ is a nonnegative $\P$-martingale starting at zero, hence is identically zero. Therefore $\tau_{K_1} = \tau_K$. If $K_1 = F_{\bar x}$, we are done.  If not, we iterate the procedure and fix another halfspace
$H_2 =\{x\in\R^d\colon a_2^\top x\ge b_2\} \neq H_1$ with  $(a_2,b_2)\in\R^d\times\R$  such that  $K_1 \subset H_2$ and  $\bar x \in \partial H_2$. Setting $K_2 = K_1 \cap \partial  H_2$ yields  $\dim(K_2) < \dim(K_1)$ and $F_{\bar x} \subset K_2$. As above, we again obtain $\tau_{K_2} = \tau_{K_1} = \tau_K$.  We proceed in the same way, but thanks to the reduction in dimension at most $d$ times, until $K_k = F_{\bar x}$ for some $k \in \{1, \ldots, d\}$. We then have $\tau_{F_{\bar x}} = \tau_{K_k} = \ldots =  \tau_{K_1} = \tau_K$, which proves the statement.
\end{proof}

\begin{lemma}\label{L_Kvzero}
Let $K$ be a convex body and consider a point $x\in K$. Then $v(x)=0$ if and only if $\dim(F_x)\le1$.
\end{lemma}

\begin{proof}
Suppose $\dim(F_x)=0$, so that $F_x=\{x\}$ is a singleton. Then $X$ leaves $F_x$ immediately under any $\P\in\Pcal_x$, that is,   $\tau_{F_x}= 0$. Suppose instead that $\dim(F_x)=1$, so that $F_x$ is a line segment. Then under any $\P\in\Pcal_x$, $X$ evolves like a one-dimensional Brownian motion along the line segment $F_x$, at least until $\tau_{F_x}$.  Thus $\essi{\P} \tau_{F_x}=0$, since $X$ reaches the endpoints of $F_x$ arbitrarily quickly with positive probability. By Lemma~\ref{L_FtauK}, we have $\essi{\P} \tau_{K}=0$. Therefore, in either case, we deduce that $v(x)=0$. For the converse direction, assume that $\dim(F_x)>1$. Then there exists a $\dim(F_x)$-dimensional closed ball $B \subset F_x$ with radius $r>0$. Since $\tau_{B} \leq \tau_{K}$, Example~\ref{Ex:190511} yields  $v(x) \geq r^2 > 0$.
\end{proof}

We now discuss continuity of the value function $v$. It was shown in Proposition~\ref{P_v_proporties}\ref{P_v_proporties_2} that the value function $v$ is upper semicontinuous. Therefore, if $v(x)=0$ at a point $x\in K$, then $v$ must be continuous at $x$. Of course, many convex bodies $K$ have boundary faces of dimension two or higher, in which case Lemma~\ref{L_Kvzero} shows that $v$ will not be zero everywhere on the boundary. Still, even in such cases, one might hope that $v$ remains continuous. Unfortunately, this is not true in general, as the following example shows.

\begin{example}\label{E_v_dicontinuous}
Let $C=\{(x,x(1-x),0)\colon x\in(0,1]\}$ denote a half-open arc in the $xy$-plane. Next let $K_0\subset\R^3$ be the closed convex hull of $\{(0,0,1), (0,0,-1)\} \cup C$. Then every point $\bar x_0\in C$ is an extreme point of $K_0$, but the origin $(0,0,0)\in\{(0,0)\}\times[-1,1]\subset K_0$ is not, despite being a limit point of $C$. Now, define $K=K_0\times[-1,1]\subset\R^4$, which is compact and convex. If $\bar x=(\bar x_0,0)\in C\times\{0\}$, then $\dim(F_{\bar x})=1$, so $v(\bar x)=0$ by Lemma~\ref{L_Kvzero}. On the other hand, the boundary face containing the origin is the square $F_0=\{(0,0)\}\times(-1,1)^2$ with $\dim(F_0)=2$, so that $v(0)>0$. Since the origin is a limit point of $C\times\{0\}$, we conclude that $v$ is not continuous on $K$, despite $K$ being convex.
\end{example}

In Example~\ref{E_v_dicontinuous}, continuity of $v$ fails because $\Fcal_1$ is not closed. One might therefore hope that continuity can be proved if $\Fcal_1,\ldots,\Fcal_d$ are closed. (Requiring $\Fcal_0$ closed should be, and is, unnecessary because $v$ is zero on all of $\Fcal_1$.) This condition indeed turns out to imply continuity. The proof iterates over the $k$-skeletons, in each step making use of the following refined version of the argument in Proposition~\ref{P:190523}. The argument is probabilistic and rests on the dynamic programming principle.

\begin{lemma}\label{L_modulus_of_v}
Let $K$ be a convex body, fix $k\in\{1,\ldots,d\}$, and assume $v|_{\cl(\Fcal_{k-1})}$ is continuous. Then there is a modulus $\omega$ such that the following holds. If $\bar x,\bar y\in\Fcal_k$, $\dim(F_{\bar x})\le\dim(F_{\bar y})$, $A$ is an affine subspace containing $F_{\bar x}$, and $Q$ is an orthogonal $d\times d$ matrix such that the map $x\mapsto Q(x-\bar x)+\bar y$ maps $A$ to $\aff(F_{\bar y})$, then
\[
v(\bar x) \le v(\bar y) + \omega(c|Q-I| + |\bar x-\bar y|),
\]
where $c=\diam(K)$ is the diameter of $K$.
\end{lemma}

\begin{proof}
Since $v$ is upper semicontinuous by Proposition~\ref{P_v_proporties}\ref{P_v_proporties_2}, since $\cl(\Fcal_{k-1})$ is compact, and since $v|_{\cl(\Fcal_{k-1})}$ is continuous by assumption, Lemma~\ref{L_unif_usc} gives a modulus $\omega$ such that
\begin{equation}\label{eq_26012019_2}
\text{$v(x) \le v(y) + \omega(|x-y|)$ for all $x\in\R^d$ and $y\in\cl(\Fcal_{k-1})$.}
\end{equation}
We now show that $\omega$ satisfies the claimed property. To this end, let $\bar x$, $\bar y$, $A$, and $Q$ be as in the statement of the lemma, and select an optimal law $\P\in\Pcal_{\bar x}$. Lemma~\ref{L_FtauK} asserts that $X(t)\in F_{\bar x}$ for all $t\le\tau_K$, $\P$-a.s. By modifying the behavior after $\tau_K$, which does not affect the optimality of $\P$, we may therefore assume that
\begin{equation}\label{eq_26012019_1}
\text{$X(t)\in A$ for all $t\ge0$, $\P$-a.s.}
\end{equation}
Consider the affine isometry $\Phi\colon A\to\aff(F_{\bar y})$ given by $\Phi(x)=Q(x-\bar x)+\bar y$. Using this isometry, define
\[
\text{$Y=\Phi(X)$ and $\theta=\inf\{t\ge0\colon Y(t)\notin \ri(F_{\bar y})\}$.}
\]
Note that $\P(\theta<\infty)=1$ by Example~\ref{Ex:190511}. Due to \eqref{eq_26012019_1}, $Y$ takes values in $\aff(F_{\bar y})$, and hence $Y(\theta)\in\rbd(F_{\bar y})\subset\Fcal_{k-1}$, $\P$-a.s. Thus by \eqref{eq_26012019_2} and monotonicity of $\omega$ we have, $\P$-a.s.,
\begin{align*}
v(X(\theta)) &\le v(Y(\theta)) + \omega(|X(\theta)-Y(\theta)|) \\
&= v(Y(\theta)) + \omega(|(I-Q)(X(\theta)-\bar x)+\bar x-\bar y|) \\
&\le v(Y(\theta)) + \omega(c|I-Q| + |\bar x-\bar y|),
\end{align*}
where $c=\diam(K)$. We now combine this with two applications of the dynamic programming principle of Proposition~\ref{P_v_proporties}\ref{P_v_proporties_3}. This is permissible because $\theta$ is $\P$-a.s.\ equal to an $\F^X$-stopping time, despite not being an $\F^X$-stopping time itself in general. We get
\begin{align*}
v(\bar x) &= \essi{\P}  \{\theta\wedge\tau_K + v(X(\theta))\bm1_{\theta\le\tau_K}\} \\
&\le \essi{\P}  \{\theta + v(Y(\theta))\} + \omega(c|I-Q| + |\bar x-\bar y|) \\
&\le v(\bar y) + \omega(c|I-Q| + |\bar x-\bar y|).
\end{align*}
In the last inequality, the application of the dynamic programming principle uses that the law of $Y$ lies in $\Pcal_{\bar y}$ due to the isometry property of $\Phi$, that $\F^Y=\F^X$, and that $\theta\le\inf\{t\ge0\colon Y(t)\notin K\}$, $\P$-a.s. This completes the proof.
\end{proof}

We now state the key propagation of continuity result, analogous to Proposition~\ref{P:190523}. Part of the proof is convenient to phrase in terms of convergence of affine subspaces. For affine subspaces $A_n$ and $A$ of $\R^d$, we say that $A_n\to A$ if $\dim(A_n)=\dim(A)$ for all large $n$, there are points $x_n\in A_n$ and $x\in A$ such that $x_n\to x$, and $A_n-x_n$ converges to $A-x$ as elements of the Grassmannian ${\rm Gr}(\dim(A),\R^d)$ of $\dim(A)$-dimensional linear subspaces of $\R^d$. In this case, there exist orthogonal $d\times d$ matrices $Q_n$ such that $Q_n\to I$ and the map $y\mapsto Q_n(y-x)+x_n$ maps $A$ to $A_n$  for $n$ sufficiently large. 
The Grassmannian is known to be compact. Therefore, whenever the affine subspaces $A_n$ contain points $x_n$ that converge to some limit, it is possible to select a convergent subsequence of the $A_n$.

\begin{lemma}\label{L_cont_induction}
Let $K$ be a convex body, fix $k\in\{1,\ldots,d\}$, and assume $v|_{\cl(\Fcal_{k-1})}$ is continuous. Then $v|_{\Fcal_k}$ is also continuous.
\end{lemma}

\begin{proof}
Since $v$ is upper semicontinuous by Proposition~\ref{P_v_proporties}\ref{P_v_proporties_2}, it suffices to show that $v|_{\Fcal_k}$ is lower semicontinuous. Since $v|_{\cl(\Fcal_{k-1})}$ is continuous by assumption, this amounts to showing that
\begin{equation}\label{eq_lsc_proof}
\text{$\bar x\in\Fcal_k$, $x_n\in\Fcal_k\setminus\cl(\Fcal_{k-1})$, $x_n\to\bar x$, $v(x_n)\to\alpha\in\R$} \quad\Longrightarrow\quad v(\bar x)\le\alpha.
\end{equation}
Let therefore $\bar x,x_n,\alpha$ be as in \eqref{eq_lsc_proof}. Define $r_n = {\rm dist}(x_n, \rbd(F_{x_n}))$. This is the radius of the largest $k$-dimensional ball centered at $x_n$ and contained in $F_{x_n}$. We consider two separate cases.

{\it Case 1:} Suppose $\liminf_{n\to\infty}r_n=0$. After passing to a subsequence, we have $r_n\to0$. Then there exist points $y_n\in\rbd(F_{x_n})$ such that $|x_n-y_n|\to0$. Thus $y_n\in\Fcal_{k-1}$ and $y_n\to\bar x$, so that $\bar x\in\cl(\Fcal_{k-1})$ and $v(y_n)\to v(\bar x)$. Moreover, applying Lemma~\ref{L_modulus_of_v} with $\bar x=y_n$, $\bar y=x_n$, $A=\aff(F_{x_n})$, and $Q=I$ then gives
\[
v(x_n) = v(y_n) + v(x_n) - v(y_n) \ge v(y_n) - \omega(|x_n-y_n|) \to v(\bar x).
\]
Thus $v(\bar x)\le\alpha$, proving \eqref{eq_lsc_proof} in this case.

{\it Case 2:} Suppose instead there exists $r>0$ such that $r_n\ge r$ for all $n$. Then each $F_{x_n}$ contains a $k$-dimensional ball $B_n$ of radius $r$ centered at $x_n$. After passing to a subsequence, we have $\aff(F_{x_n})\to A$ for some $k$-dimensional affine subspace $A$. Thus there exist orthogonal $d\times d$ matrices $Q_n$ such that $Q_n\to I$ and the affine isometry $\Phi_n\colon x\mapsto Q_n(x-\bar x)+x_n$ maps $A$ to $\aff(F_{x_n})$ for each $n$. Now, let $B\subset A$ be the $k$-dimensional ball of radius $r$ centered at $\bar x$. There is only one such ball, and we have $B_n=\Phi_n(B)$ for all $n$. For any $x\in B$ we thus have $\Phi_n(x)\in F_{x_n}\subset K$ and $\Phi_n(x)\to x$. Since $K$ is closed, it follows that $B\subset K$. Hence $B\subset F_{\bar x}$, so that $A=\aff(B)\subset\aff(F_{\bar x})$. On the other hand, $\dim(A)=k\ge\dim(F_{\bar x})$, so in fact $A=\aff(F_{\bar x})$. We now apply Lemma~\ref{L_modulus_of_v} with $\bar x$, $\bar y=x_n$, $A=\aff(F_{\bar x})$, and $Q=Q_n$ to get
\[
v(\bar x) \le v(x_n) + \omega(c|Q_n-I| + |\bar x-x_n|)
\]
with $c=\diam(K)$. Sending $n\to\infty$ yields $v(\bar x)\le\alpha$ and proves \eqref{eq_lsc_proof}.
\end{proof}

In view of Lemma~\ref{L_cont_induction}, it is of interest to know whether the $k$-skeletons of a given convex body are closed. For some values of $k$, closedness is automatic.

\begin{lemma}\label{L_skeletons_closed}
Let $K$ be a convex body. Then $\Fcal_d$, $\Fcal_{d-1}$, and $\Fcal_{d-2}$ are closed.
\end{lemma}

\begin{proof}
Both $\Fcal_d=K$ and $\Fcal_{d-1}=\partial K$ are closed. To see that $\Fcal_{d-2}$ is closed, assume for contradiction that there is a point $\bar x\in\cl(\Fcal_{d-2})\setminus\Fcal_{d-2}$. Then $\bar x$ lies in $\partial K$ but not in $\Fcal_{d-2}$, so must lie in the relative interior of a $(d-1)$-dimensional boundary face $F$. But then $\bar x$ admits an open neighborhood contained in $\ri(F)\cup\interior(K)\cup K^c$, and therefore cannot lie in the closure of $\Fcal_{d-2}$. This contradiction finishes the proof.
\end{proof}

Here is the main result of this section.

\begin{theorem}\label{T_190603}
Let $K$ be a convex body with $\Fcal_k$ closed for $1\le k\le d-3$. Then $v|_K$ is continuous.
\end{theorem}

\begin{proof}
By Lemma~\ref{L_Kvzero}, $v$ vanishes on $\Fcal_1$ and is therefore continuous there. Continuity on $K$ now follows by repeated application of Lemma~\ref{L_cont_induction}, making use of the closedness hypothesis on the $(k-1)$-skeletons for $1 \leq k-1\le d-3$, and Lemma~\ref{L_skeletons_closed} for $k-1\ge d-2$.
\end{proof}

As an immediate corollary, several interesting cases are covered.

\begin{corollary}\label{C_examplesK}
Each of the following conditions implies that $v|_K$ is continuous.
\begin{enumerate}
\item\label{C_examplesK_1} All boundary faces of $K$ have dimension zero or one.
\item\label{C_examplesK_2} $\dim(K)\in\{0,1,2,3\}$.
\item\label{C_examplesK_3} $K$ is a (convex) polytope.
\end{enumerate}
\end{corollary}

\begin{proof}
\ref{C_examplesK_1} and \ref{C_examplesK_2} are immediate from Theorem~\ref{T_190603}. As for \ref{C_examplesK_3}, if $K$ is a polytope, then $\Fcal_k$ is closed for all $k=0,\ldots,d$; see e.g.\ \citet[Theorem~2.3]{pap_77} and the subsequent discussion.
Now apply Theorem~\ref{T_190603}.
\end{proof}

Thanks to Lemma~\ref{L_Kvzero}, if $v|_K$ is continuous then $\Fcal_1$ is necessarily closed.  
Are  the other $k$-skeletons also closed in this case?    If $d=4$ the answer is yes thanks to Lemma~\ref{L_skeletons_closed}.
In general the answer is no, as shown in Example~\ref{Ex:190526} below. Continuity of the value function therefore cannot be used to characterize closedness of the $k$-skeletons.

\begin{example}  \label{Ex:190526}
Recall the set $K_0 \subset \R^3$ from Example~\ref{E_v_dicontinuous}.
Let $K' = K_0 \times \R^2 \subset \R^5$ and set
\[
	K = K' \cap \{(x,y,z,u,w) \in \R^5: z^2 +  u^2 + w^2 \leq 1\}.
\]
Then $K$ is a convex body, and one can check that
\begin{align*}
\Fcal_1 &=  \{(x,y,z,u,w) \in \R^5: (x,y,z) \in \partial K_0, z^2 +  u^2 + w^2 = 1\};\\
\Fcal_2 &=   \{(x,y,z,u,w) \in \R^5: (x,y,z) \in K_0, z^2 +  u^2 + w^2 = 1\}\\
&\quad \cup \{(x,y,0,u,w) \in \R^5: x \in (0,1], y = x(1-x),  u^2 + w^2 \leq 1\}; \\
\cl(\Fcal_1) &= \Fcal_1; \\
\cl(\Fcal_2) &= \Fcal_2 \cup \{(0,0,0,u,w): u^2 + w^2 \leq 1\}.
\end{align*}
Since $\Fcal_1$ is closed, Lemmas~\ref{L_Kvzero} and~\ref{L_cont_induction} imply that $v|_{\Fcal_2}$ is continuous. We claim that $v|_{\cl(\Fcal_2)}$ is also continuous, but since $\Fcal_2\ne\cl(\Fcal_2)$ we must argue this directly at the remaining points of $\cl(\Fcal_2)$. Consider therefore such a point $(0,0,0,u,w) \in \cl(\Fcal_2)$ with $u^2 + w^2 \leq 1$ and an approximating sequence of points $(x_n, y_n, z_n, u_n, w_n)\in \cl(\Fcal_2)$. Then we know that 
$v(x_n, y_n, z_n, u_n, w_n) \geq 1 - z_n^2-u_n^2 -w_n^2$ for all $n \in \N$ thanks to Example~\ref{Ex:190511}. Thus by upper semicontinuity, we have
\begin{align*}
1-u^2-w^2 = v(0,0,0,u,w) &\ge \limsup_n v(x_n, y_n, z_n, u_n, w_n) \\
&\ge \liminf_n v(x_n, y_n, z_n, u_n, w_n) \ge 1-u^2-w^2.
\end{align*}
This yields continuity of $v|_{\cl(\Fcal_2)}$. Hence by Lemma~\ref{L_cont_induction}, $v|_{\Fcal_3}$ is continuous. Lemma~\ref{L_skeletons_closed} yields that $\Fcal_3$, $\Fcal_4$, $\Fcal_5=K$ are closed. Repeating the previous argument shows that $v|_K$ is continuous, even though $\Fcal_2$ is not closed.
\end{example}

\section{Smooth value functions}\label{S_v_smooth}

The goal of this section is to prove Theorem~\ref{T_smooth_case}. Let us first introduce some terminology. Let $K$ be a convex body. We say that a function $f$ lies in $C^2(K)$ if $f|_K$ is continuous, and the restriction $f|_{\ri(F)}$ to the relative interior of any face $F$ of $K$ lies in $C^2(\ri(F))$, understood in the usual sense of twice continuous differentiability on the $\dim(F)$-dimensional open set $\ri(F)\subset\aff(F)$. The gradient and Hessian computed relative to this set are then denoted $\nabla_K f(x)=\nabla(f|_{\ri(F)})(x)\in\aff(F-x)$ and $\nabla_K^2f(x)=\nabla^2(f|_{\ri(F)})(x)$ for any $x\in\ri(F)$. Thus $\nabla_K f(x)$ and $\nabla_K^2f(x)$ are the projections of $\nabla f(x)$ and $\nabla^2f(x)$ onto $\aff(F-x)$, whenever the latter exist.  A critical point of $f$ in $F$ is a point $x\in\ri(F)$ where $\nabla_K f(x)=0$. 

To prove Theorem~\ref{T_smooth_case}, it is enough to prove the following.

\begin{theorem}\label{T_smooth_case_second}
Let $d\ge2$ and let $K$ be a convex body with at most countably many faces. Assume the value function $v$ lies in $C^2(K)$. Assume also that in each face $F$ of dimension at least two, either $v$ has no critical point, or $v$ has one single critical point which additionally is a maximum. Then for every $\bar x\in K$ there is an optimal solution $\P\in\Pcal_{\bar x}$ under which $v(X(t))=v(\bar x)-t$ for all $t<\tau_K$.
\end{theorem}

Observe that the assumption that $v$ lies in $C^2(K)$ immediately implies that $v|_{\aff(F)}$ is a classical solution of \eqref{PDE_1} in $\ri(F)$ away from the critical point, for every face $F$ of dimension at least two; just use $v$ itself as test function in the definition of viscosity sub- and supersolutions.

The proof of Theorem~\ref{T_smooth_case_second} proceeds by first constructing solution laws $\P$ under which $X$ behaves in the desired manner while inside any given face $F$ of $K$. Then these laws are pasted together as $X$ reaches ever lower-dimensional faces, until it leaves $K$. To implement this idea, for any face $F$ of $K$ with $\dim(F)\ge2$ and any point $x\in \ri(F)$, we define
\[
\Pcal_x^* = \{\P\in\Pcal_x\colon \text{$v(X(t))=v(x)-t$ for all $t<\tau_{\ri(F)}$ and $X(\tau_{\ri(F)})\in\rbd(F)$}\},
\]
where $\tau_{\ri(F)}=\inf\{t\ge0\colon X(t)\notin\ri(F)\}$ is the first time $X$ leaves the relative interior of $F$. For points $x\in K^c \cup \Fcal_1$, we somewhat arbitrarily set $\Pcal_x^*=\Pcal_x$.

Let now the hypotheses of Theorem~\ref{T_smooth_case_second} be in force. Our first goal is to prove that $\Pcal_x^*$ is nonempty for every $x \in \R^d$. This rests on the following construction of a martingale with increments in the kernel of a given location-dependent matrix.

\begin{lemma}\label{L_SDE_existence}
Let $\Ocal\subset\R^d$ be open and let $H\colon\R^d\to\S^d$ be a locally bounded measurable map such that $H|_\Ocal$ is continuous and $\rk H(x)\le d-1$ for all $x\in\Ocal$. For every $\bar x\in\Ocal$ there exists a continuous martingale $Y$ with $Y(0)=\bar x$ and $\tr\langle Y\rangle(t)\equiv t$ such that
\begin{equation}\label{L_SDE_existence_1}
\int_0^t H(Y(s))d\langle Y\rangle(s) = 0, \quad t<\tau_\Ocal,
\end{equation}
where $\tau_\Ocal=\inf\{t\ge0\colon Y(t)\notin\Ocal\}$.
\end{lemma}

\begin{proof}
Define
\[
a(x) = \begin{cases}{\displaystyle\frac{I - H(x)^+ H(x)}{d-\rk H(x)}}, & x\in\Ocal\\ d^{-1} I,&x\notin\Ocal \end{cases}
\]
where $H(x)^+$ is the Moore--Penrose generalized inverse of $H(x)$. Thus if the spectral decomposition of $H(x)$ is $H(x)=Q\diag(\lambda_1,\ldots,\lambda_r,0,\ldots,0) Q^\top$ with $\lambda_i\ne0$ for $i=1,\ldots,r$, then $H(x)^+H(x)= Q \diag(1,\ldots,1,0,\ldots,0) Q^\top$, where the diagonal matrix contains $r$ ones. It follows that
\[
\text{$a(x)\succeq0$ and $\tr(a(x))=1$ for all $x\in\R^d$, and $H(x)a(x)=0$ for all $x\in\Ocal$.}
\]

Unless $H$ has constant rank on $\Ocal$, $a$ is not continuous on $\Ocal$. Consider therefore mollifications
\[
a_n(x) = \int_{\R^d} \varphi_n(x-y) a(y) dy,
\]
where $\varphi_n(x)=n^d\varphi(nx)$ for a positive mollifier $\varphi$ supported on the centered unit ball. Then $a_n$ is continuous, positive semidefinite, and has unit trace. Thus there exist weak solutions $Y_n$ of the SDEs
\[
dY_n(t) = a_n(Y_n(t))^{1/2} dW(t), \quad Y_n(0)=\bar x,
\]
where the positive semidefinite square root is understood, and $W$ is $d$-dimensional Brownian motion. The law of $Y_n$ lies in $\Pcal_{\bar x}$ for each $n$, so Proposition~\ref{P_v_proporties}\ref{P_v_proporties_1} shows that after passing to a subsequence, $Y_n\Rightarrow Y$ for some limiting process $Y$ whose law again lies in $\Pcal_{\bar x}$. Since $\langle Y_n\rangle(t)=\int_0^t a_n(Y_n(s))ds$ and the $a_n$ are uniformly bounded, after passing to a further subsequence we actually have $(Y_n,\langle Y_n\rangle)\Rightarrow(Y,Q)$ in the space $C(\R_+,\R^d\times\S^d)$ for some process $Q$. Since $Y_nY_n^\top-\langle Y_n\rangle$ is a martingale for each $n$, and using the uniform bound on the quadratic variations, we may pass to the limit to deduce that $YY^\top-Q$ is also martingale, and hence $Q=\langle Y\rangle$. Furthermore, by Skorohod's representation theorem (see \citet[Theorem~6.7]{B:99}), we may assume that the $(Y_n,\langle Y_n\rangle)$ and $(Y,\langle Y\rangle)$ are defined on a common probability space $(\Omega',\Fcal',\P')$ and that, almost surely, $(Y_n,\langle Y_n\rangle)\to(Y,\langle Y\rangle)$ in $C(\R_+,\R^d\times\S^d)$, that is, locally uniformly.

We now verify \eqref{L_SDE_existence_1}. We first claim that
\begin{equation}\label{eq_L_SDE_existence_pf1}
\text{if $x\in\Ocal$ and $x_n\to x$ then $H(x_n)a_n(x_n)\to0$.}
\end{equation}
To prove this, note that
\[
H(x_n)a_n(x_n) = \int_{\R^d} \varphi_n(x_n-y)H(x_n)a(y)dy = \int_{\R^d} \varphi_n(x_n-y)(H(x_n)-H(y))a(y)dy.
\]
Since $a$ is bounded  and the restriction $H|_\Ocal$ is continuous, arguing component by component, we see that the right-hand side converges to zero.
This proves \eqref{eq_L_SDE_existence_pf1}. Now pick $t<\tau_{\Ocal}$. Then $Y(s)\in\Ocal$ for all $s\le t$. Since $(Y_n,\langle Y_n \rangle)\to(Y,\langle Y\rangle)$ locally uniformly, the bounded convergence theorem and \eqref{eq_L_SDE_existence_pf1} yield that
\[
\int_0^t H(Y_n(s)) d\langle Y_n\rangle(s) = \int_0^t H(Y_n(s)) a_n(Y_n(s)) ds \to 0.
\]
On the other hand, the left-hand side converges to $\int_0^t H(Y(s)) d\langle Y\rangle(s)$. This yields \eqref{L_SDE_existence_1} and completes the proof of the lemma.
\end{proof}

\begin{remark}
An examination of the proof of Lemma~\ref{L_SDE_existence} shows that the process $Y$ is of the form $dY_t=\sigma_t dW'_t$ for some Brownian motion $W'$, where $\sigma_t=a(Y_t)^{1/2}$ for all $t$ such that $a$ is continuous at $Y_t$. By properties of the Moore--Penrose inverse, $a$ is continuous except on the boundaries of the sets $\{x\in\Ocal\colon \rk H(x)=r\}$, $r=0,\ldots,d-1$ and on $\partial \Ocal$. Thus if $Y$ can be shown to spend zero time in these sets, it is a bona fide weak solution of $dY_t=a(Y_t)^{1/2}dW'_t$.
\end{remark}

\begin{proposition}\label{P_Pxstar_nonempty}
Continue to assume $v\in C^2(K)$. Then $\Pcal_{\bar x}^*$ is nonempty for every $\bar x\in\R^d$.
\end{proposition}

\begin{proof}
If $\bar x\in K^c\cup\Fcal_1$, then $\Pcal_{\bar x}^*=\Pcal_{\bar x}$ and the statement is obvious. Below we prove the statement for $\bar x\in\interior(K)$; the case $\bar x\in\ri(F)$ for a face $F$ with $\dim(F)\ge2$ is identical since all considerations are then restricted to $\aff(F)$. So suppose $\bar x\in\interior(K)$ and, initially, also that $\bar x$ is not the maximizer of $v$ over $K$; in particular $\bar x$ is not a critical point in $K$.

Since $v$ lies in $C^2(K)$, it is a classical solution of \eqref{PDE_1} in $\interior(K)$ away from the critical point. As explained in Section~\ref{S_main}, an alternative form of this equation at non-critical points is \eqref{eq_PDE_lambda_intro}. That is,
\begin{equation}\label{eq_PDE_lambda}
\lambda_{\rm min}\left(-\frac12 P_{\nabla v(x)} \nabla^2v(x) P_{\nabla v(x)}, \nabla v(x)\right) = 1,
\end{equation}
where $\lambda_{\rm min}(A,p)$ denotes the smallest eigenvalue of $A$ corresponding to an eigenvector orthogonal to $p$, and
\[
P_{\nabla v(x)} = I - \frac{\nabla v(x)\nabla v(x)^\top}{|\nabla v(x)|^2}.
\]
Let $\Ocal = \{x\in\interior(K)\colon \nabla v(x)\ne0\}$ be the set of non-critical points in $\interior(K)$. Define
\[
H(x) = \frac12 P_{\nabla v(x)} \nabla^2v(x) P_{\nabla v(x)} + I, \quad x\in\Ocal,
\]
and arbitrarily set $H(x)=0$ for $x\notin\Ocal$. It is clear that $H$ is locally bounded measurable and that $H|_{\Ocal}$ is continuous. Moreover, the equation \eqref{eq_PDE_lambda} satisfied by $v$ implies that $H(x)$ is singular, i.e.\ $\rk H(x)\le d-1$, for all $x\in\Ocal$. We may thus apply Lemma~\ref{L_SDE_existence} to obtain a martingale $Y$ whose law we denote by $\P$. Clearly $\P\in\Pcal_{\bar x}$, and due to  \eqref{L_SDE_existence_1} we have
\[
\text{$H(X(t))a(t) = 0$ on $[0,\tau_\Ocal)$, $\P$-a.s.,}
\]
where $a(t)$ satisfies $\langle X\rangle = \int_0^\fdot a(s) ds$ and $\tr(a(t))=1$, and $\tau_\Ocal=\inf\{t\ge0\colon X(t)\notin\Ocal\}$. As a consequence, omitting the argument $X(t)$ for readability, we have for $t<\tau_\Ocal$ that
\[
0 = \nabla v^\top H a(t) = \frac12 \nabla v^\top P_{\nabla v}\nabla^2v P_{\nabla v}a(t) + \nabla v^\top a(t) = \nabla v^\top a(t).
\]
We thus have $\nabla v^\top a(t)=0$, which yields $P_{\nabla v}a(t)P_{\nabla v}=a(t)$. Consequently,
\[
0 = \tr(H a(t)) = 1 + \tr\left(\frac12 P_{\nabla v} \nabla^2v P_{\nabla v}a(t)\right) = 1 + \frac12  \tr( a(t)\nabla^2v).
\]
An application of It\^o's formula now gives
\[
dv(X(t)) = \nabla v(X(t))^\top dX(t) + \frac12  \tr( a(t)\nabla^2v(X(t)))dt = -dt.
\]
These computations are valid for $t<\tau_\Ocal$, so we deduce that $v(X(t)) = v(\bar x)-t$ for $t<\tau_\Ocal$. In particular, $X(t)$ will not attain a critical point before $\tau_\Ocal$, so in fact $\tau_\Ocal=\tau_{\interior(K)}$, the first exit time from $\interior(K)$. Moreover, at the exit time, we have $X(\tau_{\interior(K)})\in\partial K$. This shows that $\P\in\Pcal_{\bar x}^*$, as desired.

The case where $\bar x$ is a critical point still remains. In this case, we select points $x_n\in\interior(K)\setminus\{\bar x\}$ with $x_n\to\bar x$, and let $\P_n\in\Pcal_{x_n}^*$. In particular, the laws $\Q_n=(\fdot-x_n)_*\P_{x_n}$ lie in $\Pcal_0$, which is  compact by Proposition~\ref{P_v_proporties}\ref{P_v_proporties_1}. The $\Q_n$ are thus subsequentially convergent toward some $\Q\in\Pcal_0$. Along this subsequence, the $\P_n$ converge to $\P=(\fdot+\bar x)_*\Q \in \Pcal_{\bar x}$. Lemma~\ref{L_CF_closed} below shows that the properties $v(X(t))=v(X(0))-t$ for all $t<\tau_{\interior(K)}$ and $X(\tau_{\interior(K)})\in \partial K$ if $\tau_{\interior(K)} < \infty$ carry over to weak limits. This shows that $\P\in\Pcal_x^*$, and completes the proof of the proposition.
\end{proof}

We now turn to the task of pasting solutions together as $X$ reaches ever lower-dimensional faces of $K$. This uses a measurable selection of laws from $\Pcal_x^*$, which in turn requires suitable closedness properties of these sets. The following closedness result was already used in the proof of Proposition~\ref{P_Pxstar_nonempty}.

\begin{lemma}\label{L_CF_closed}
Let $F$ be a face of $K$ and write $\sigma=\tau_{\ri(F)}$ for brevity. Then the set
\[
C_F = \{\omega\in\Omega\colon \text{$v(\omega(t))=v(\omega(0))-t$ $\forall\,t<\sigma(\omega)$, and $\omega(\sigma(\omega))\in\rbd(F)$ if $\sigma(\omega)<\infty$}\}
\]
is closed in $\Omega$. As a consequence, $\{\P\in\Pcal(\Omega)\colon \P(C_F)=1\}$ is closed in $\Pcal(\Omega)$. The same conclusion holds if $v$ is only known to be continuous, not necessarily $C^2(K)$.
\end{lemma}

\begin{proof}
It suffices to prove that $C_F$ is closed, as the second statement then follows from the Portmanteau lemma. Pick $\omega_n\in C_F$ with $\omega_n\to\omega$ in $\Omega$. Define $T=\liminf_n\sigma(\omega_n)\in[0,\infty]$, and pass to a subsequence to get $T=\lim_n\sigma(\omega_n)$. Then $v(\omega_n(t))=\omega_n(0)-t$ and $\omega_n(t)\in F$ if $n$ is sufficiently large, for all $t<T$. Since $v$ is continuous and $F$ closed, we get $v(\omega(t))=v(\omega(0))-t$ and $\omega(t)\in F$ for all $t<T$. Provided $\sigma(\omega)\le T$, this implies $\omega\in C_F$ and proves closedness. If $T=\infty$ then of course $\sigma(\omega)\le T$. If $T<\infty$, then by definition of $T$ we have $\sigma(\omega_n)\le T+\varepsilon$ for any $\varepsilon>0$ and all large $n$. Thus $\min_{t\le T+\varepsilon}{\rm dist}(\omega_n(t),\rbd(F))=0$ for all large $n$. By continuity we get $\min_{t\le T+\varepsilon}{\rm dist}(\omega(t),\rbd(F))=0$, and hence $\sigma(\omega)\le T+\varepsilon$. Since $\varepsilon>0$ was arbitrary, this yields $\sigma(\omega)\le T$ as required.
\end{proof}

The following lemma produces the required measurable selection. This is actually the only step that uses that $K$ has countably many faces. If the lemma could be established without assuming this, the assumption could be dropped from Theorem~\ref{T_smooth_case_second} (and Theorem~\ref{T_smooth_case}). In fact, the current proof works for the more general situation where $K$ has countably many faces of dimension two and higher, and arbitrarily many faces of dimension zero and one.

\begin{lemma}\label{L_Pxstar_mb}
Assume $K$ has countably many faces, and continue to assume $v\in C^2(K)$. Them there is a measurable map $x\mapsto\P_x$ from $\R^d$ to $\Pcal(\Omega)$ such that $\P_x\in\Pcal_x^*$ for all $x$.
\end{lemma}

\begin{proof}
We apply the selection theorem of Kuratowski and Ryll-Nardzewski; see \citet[Theorem~18.13]{ali_bor_06}. This requires that the set-valued map $x\mapsto\Pcal_x^*$ be weakly measurable with nonempty closed values. By Proposition~\ref{P_Pxstar_nonempty}, $\Pcal_x^*$ is nonempty for all $x$. For $x\in K^c\cup\Fcal_1$, $\Pcal_x^*=\Pcal_x$ is closed (even compact) by Proposition~\ref{P_v_proporties}\ref{P_v_proporties_1}. If $F$ is a face of $K$ with $\dim(F)\ge2$ and $x\in\ri(F)$, then
\[
\Pcal_x^*=\Pcal_x\cap\{\P\in\Pcal(\Omega)\colon \P(C_F)=1\},
\]
which is closed by Lemma~\ref{L_CF_closed}. So $\Pcal_x^*$ is closed for all $x$. 

We now argue weak measurability, initially for the map $x\mapsto\Pcal_x$. We must show that for every open subset $U\subset\Pcal(\Omega)$, the set $\{x\in\R^d\colon \Pcal_x\cap U\ne\emptyset\}$ is measurable; see \citet[Definition~18.1]{ali_bor_06}. But since $\Pcal_x=(\fdot+x)_*\Pcal_0$, the condition $\Pcal_x\cap U\ne\emptyset$ means that there exists $\P\in\Pcal_0$ such that $(\fdot+x)_*\P\in U$. If this holds for some $x\in\R^d$, then it also holds for all $y$ in a neighborhood of $x$ since $U$ is open and $x\mapsto(\fdot+x)_*\P$ is continuous. Thus $\{x\in\R^d\colon \Pcal_x\cap U\ne\emptyset\}$ is actually open, and in particular measurable. So $x\mapsto\Pcal_x$ is weakly measurable. 

Furthermore, the set-valued map $x\mapsto\varphi(x)$ specified by $\varphi(x)=\Pcal(\Omega)$ for $x\in K^c\cup\Fcal_1$ and $\varphi(x)=\{\P\in\Pcal(\Omega)\colon \P(C_F)=1\}$ for $x\in\ri(F)$ is constant on $K^c$ and on each face of $K$. Since $K$ has countably many faces, we deduce that $x\mapsto\varphi(x)$ is weakly measurable. By \citet[Lemma~18.4(3)]{ali_bor_06}, it now follows that $x\mapsto\Pcal_x^*=\Pcal_x\cap\varphi(x)$ is weakly measurable, as required.
\end{proof}

\begin{proof}[Proof of Theorem~\ref{T_smooth_case_second}]
For $k=2,\ldots,d$, define $U_k=\Fcal_k\setminus\Fcal_{k-1}$. Equivalently, $U_k$ is the (possibly empty) union of the relative interiors of all $k$-dimensional faces of $K$. We work on the $d$-fold product $\Omega^d=C(\R_+,\R^d)^d$ of the canonical path space, and let $(W,Y^2,\ldots,Y^d)$ be the $(\R^d)^d$-valued coordinate process. Let $x\mapsto\P_x\in\Pcal_x^*$ be the measurable map given by Lemma~\ref{L_Pxstar_mb}; we will use it to specify the law of $Y^2,\ldots,Y^d$. Define random times $\tau_{k-1}=\inf\{t\ge0\colon Y^k(t)\notin U_k\}$. Given $\bar x\in K$, let $Y^d$ have law $\P_{\bar x}$. Next, if the law of $(Y^d,\ldots,Y^k)$ has been specified for $k\ge3$, then specify the law of $Y^{k-1}$ to be conditionally independent of $(Y^d,\ldots,Y^k)$ given $Y^k(\tau_{k-1})$, which is finite almost surely, with law $Y^{k-1}\sim\P_{Y^k(\tau_{k-1})}$. That is, the regular conditional distribution of $Y^{k-1}$ given $Y^k(\tau_{k-1})=y$ is $\P_y$. This procedure specifies the law of $Y^2,\ldots,Y^d$. Finally, let $W$ have the law of an independent standard $d$-dimensional Brownian motion. Now, set $\tau_d=0$ and define a process $Y$ by
\[
Y(t) = Y^k(t-(\tau_k+\cdots+\tau_d)), \quad t\in[\tau_k+\cdots+\tau_d,\tau_{k-1}+\cdots+\tau_d),
\]
for $k=2,\ldots,d$, and
\[
Y(t) = Y^2(\tau_1)+d^{-1/2}W(t-(\tau_1+\cdots+\tau_d)), \quad t\ge \tau_1+\cdots+\tau_d.
\]
Thus $Y$ first follows the dynamics of $Y^d$ while in the interior of $K$ (a possibly empty time interval); then $Y$ follows the dynamics of $Y^{d-1}$ while inside the relative interior of a $(d-1)$-dimensional face, and so on, until it reaches a face of dimension zero or one. From that point onwards, it follows a Brownian motion, scaled so that the quadratic variation has unit trace. Since the law of each $Y^k$ is chosen from the sets $\Pcal_x^*$, it is straightforward but somewhat tedious to make this intuitive description rigorous. One also finds that $Y$ is a continuous martingale, starting at $Y(0)=\bar x$ and with $\tr\langle Y\rangle(t)\equiv t$, and (using that $v|_K$ is continuous) such that $v(Y(t))=v(\bar x)-t$ for all $t<\tau_{K\setminus\Fcal_1}=\inf\{t\ge0\colon Y(t)\notin K\setminus\Fcal_1\}$. Moreover, $Y$ does not leave $K$ before reaching $\Fcal_1$, but then leaves $K$ immediately since its dynamics switches to that of a scaled standard Brownian motion in $d\ge2$ dimensions. In particular, $\tau_{K\setminus\Fcal_1}=\tau_K$, and we have $\tau_K=v(\bar x)-v(Y(\tau_K))=v(\bar x)$, using also that $v=0$ on $\Fcal_1$. The law $\P$ of $Y$ is therefore the required optimal law.
\end{proof}

\bibliographystyle{plainnat}
\bibliography{bibl}

\end{document}